\documentclass[12pt]{amsart}
\usepackage{verbatim}
\usepackage[mathscr]{eucal}
\usepackage{amscd}
\usepackage{amsthm}
\usepackage{enumerate}
\usepackage{comment}
\usepackage{url}
\usepackage{color}
\usepackage{mathtools,amssymb}
\usepackage[titletoc]{appendix}
\usepackage{hyperref}
\usepackage[margin=1in]{geometry}

\newtheorem{theorem}{Theorem}[section]
\newtheorem{lemma}[theorem]{Lemma}
\newtheorem{corollary}[theorem]{Corollary}
\newtheorem{proposition}[theorem]{Proposition}

\theoremstyle{definition}

\newtheorem{definition}[theorem]{Definition}
\newtheorem{remark}[theorem]{Remark}

\newtheorem{question}[theorem]{Question}
\newtheorem{problem}[theorem]{Problem}

\def\R{\mathbb{R}}

\def\cF{\mathcal{F}}

\def\e{\epsilon}

\newcommand{\calI}{\mathcal{I}}
\newcommand{\calJ}{\mathcal{J}}

\newcommand{\calU}{\mathcal{U}}

\newcommand{\emp}{\raisebox{1pt}{{$\scriptstyle\langle\makebox[.02in]{}\rangle$}}}
\newcommand{\semp}{\raisebox{0.5pt}{{$\scriptscriptstyle\langle\makebox[.01in]{}\rangle$}}}
\newcommand{\conc}{\raisebox{4pt}{{$\scriptstyle\frown$}}}
\newcommand{\sconc}{\raisebox{2pt}{{$\scriptscriptstyle\frown$}}}

\newcommand{\iseq}{\trianglelefteq}
\newcommand{\bn}[1]{{\prescript{\raisebox{1pt}{$\scriptscriptstyle <$}#1}{}{2}}}
\newcommand{\bne}[1]{{\prescript{\raisebox{1pt}{$\scriptscriptstyle \leq$}#1}{}{2}}}
\newcommand{\bl}[1]{{\prescript{#1}{}{2}}}
\newcommand{\inv}{^{\text{-}1}}

\newcommand{\seq}{\subseteq}

\newcommand{\clqedsym}{\dashv_{\text{\scriptsize claim}}}

\newcommand{\clqed}{%
  \ifmmode
    \eqno{\clqedsym}%
  \else
    \hfill$\clqedsym$%
  \fi
}

\renewcommand{\Im}{\operatorname{Im}}

\newcommand{\dotminus}{ 
\!\!\buildrel\textstyle~.\over{\hbox{ 
\vrule height3pt depth0pt width0pt}{\smash-} 
}}

\numberwithin{equation}{theorem}

\usepackage{refcount}

\newcounter{savedequation}

\newenvironment{delayedproof}[1]
  {%
    \setcounter{savedequation}{\value{equation}}%
    \setcounter{equation}{0}%
    \renewcommand{\theequation}{\getrefnumber{#1}.\arabic{equation}}%
    \begin{proof}[\textnormal{\textbf{Proof of Theorem~\ref{#1}}}]%
  }
  {%
    \end{proof}%
    \setcounter{equation}{\value{savedequation}}%
  }

\title{Encoding orders and trees in real-valued functions}

\author{G. Conant}
\address{Department of Mathematics, Statistics, and Computer Science\\
University of Illinois Chicago
}
\email{gconant@uic.edu}

\author{C. Terry}
\address{Department of Mathematics, Statistics, and Computer Science\\
University of Illinois Chicago
}
\email{caterry@uic.edu}

\date{July 23, 2026}

\thanks{GC was partially supported by an NSF grant (DMS-2452816); CT was partially supported by an NSF CAREER Award (DMS-2115518) and a Sloan Research Fellowship.}

\begin{document}
\begin{abstract}
We prove  function-theoretic analogues of a quantitative result of Hodges \cite{hodges} on extracting the order property from a sufficiently large $2$-tree coded in a binary relation. Similar analogues for functions were previously obtained by Daskalakis and Golowich \cite{DaskGo} and by Anderson and Benedikt \cite{AndBen}. These results are from statistical learning theory, where $2$-trees are captured by \emph{sequential fat-shattering dimension}, and the order property is controlled by various notions of ``thresholds". Our first main result (Theorem \ref{thm:hodgesfn1}) focuses on extracting a less restrictive kind of threshold from a tree, and  yields significantly better bounds compared to what can be obtained from earlier results focusing on more restrictive versions. Part of the motivation for Theorem \ref{thm:hodgesfn1} lies in our companion paper \cite{CT-QSAR}, where this theorem is used to obtain efficient bounds in quantitative regularity lemmas for ``stable functions". Here will use Theorem \ref{thm:hodgesfn1} to reprove a result from \cite{AndBen} in a stronger form and with improved bounds. We also use Theorem \ref{thm:hodgesfn1} to  prove an at most double-exponential bound on dual sequential fat-shattering, which resolves an open problem from \cite{DaskGo}. In our second main result (Theorem \ref{thm:tight}), we give a new proof of a result from \cite{DaskGo} on extracting ``tight thresholds" from large sequential fat-shattering dimension, with improved bounds. This resolves another open problem in \cite{DaskGo} related to correcting the proof of a result claimed by Jung, Kim, and Tewari \cite{JuKiTe}. 
\end{abstract}

\maketitle

\section{Introduction}

\subsection{Summary}

In 1971, Shelah established a fundamental correspondence between orders and trees coded in first-order formulas (see \cite[Theorem 2.9]{Sh10},  \cite[Theorem II.2.2]{shelah}). The  proof  relied on model-theoretic compactness and a theorem in infinitary combinatorics due to Erd\H{o}s and Makkai \cite{ErdMak}. 
A direct  proof with explicit quantitative bounds was later obtained by Hodges \cite{hodges} using finite combinatorics. We state this version below in Theorem \ref{thm:hodges}. 

This correspondence has important consequences in several fields. In model theory, it connects the absence of the order property with boundedness of Shelah’s local 2-rank, giving a rank-theoretic formulation of stability that underlies  definability of types. In graph theory, it plays a fundamental role in the Malliaris-Shelah stable regularity lemma \cite{MS}, which sparked a remarkably fruitful line of interaction  between model theory and combinatorics. Finally, in statistical learning theory, this correspondence establishes an equivalence between bounded threshold dimension and bounded Littlestone dimension, as formulated explicitly in \cite{ALMM} with precursors in \cite{Bhas,ChaFr2}. 

In this paper, we will prove  quantitative extensions of the order-tree correspondence to the setting of bounded real-valued functions. As  discussed in Appendix \ref{ss:2rankcont}, the existence of such a correspondence is implicit in   early work in  continuous logic on a suitable adaptation of Shelah 2-rank \cite{BYscat,BBHU,BYU}. However, this approach only produces a qualitative relationship, in analogy to Shelah's work in the discrete setting prior to the paper by Hodges. Quantitative results emerged more recently in statistical learning theory, e.g., \cite{AndBen,DaskGo,JuKiTe}, which focus on the relationship between fat-threshold dimension and sequential fat-shattering dimension (the real-valued generalization of Littlestone dimension). 

Our first main result (Theorem \ref{thm:hodgesfn1}) will be an extension of this work with significantly improved bounds. As a direct consequence we will improve the known bounds on the relationship between the sequential fat-shattering dimension of dual classes (Corollary \ref{cor:dualbound}), solving an open problem from \cite{DaskGo}. In our second main result (Theorem \ref{thm:tight}), we will use an elaborate modification of the proof of Theorem \ref{thm:hodgesfn1} to improve the bounds in a result of Daskalakis and Golowich \cite{DaskGo} on extracting ``tight thresholds" from  large sequential fat-shattering dimension. This will resolve another open problem in \cite{DaskGo} related to a correct proof of a result originally claimed by Jung, Kim, and Tewari \cite{JuKiTe}. 

This paper also serves as a companion to \cite{CT-QSAR}, in which we prove  regularity lemmas for ``stable functions" with explicit quantitative bounds, along the lines of Malliaris and Shelah's results for graphs \cite{MS}. Our first main result here (Theorem \ref{thm:hodgesfn1})  will be a crucial ingredient in that paper for obtaining bounds that are polynomial in the error parameter.

\subsection{Trees}
 
In this subsection, we define a  notion of trees coded in binary $[0,1]$-valued functions. 
This definition will require the following  standard notation for the tree structure on binary strings. 
Given $t\geq 1$, let $\bl{t}$ denote $\{0,1\}^t$. We view elements of $\bl{t}$ as binary strings of length $t$. We let $2^0=\{\emp\}$, where $\emp$ denotes the empty string, which by convention has length $0$.  Given $t>0$, let
\[
\textstyle\bn{t}=\bigcup_{0\leq i<t}\bl{i}.
\]
For $t\geq 0$, we also use $\bne{t}$ to denote $\bn{t+1}$.
For $s,t\geq 0$, $\sigma \in \bl{s}$, and $\mu \in \bl{t}$, let $\sigma \conc \mu\in \bl{s+t}$ denote the concatenation of $\sigma$ and $\mu$. 
We write $\sigma\iseq \tau$ if $\sigma$ is an initial segment of $\tau$, i.e., if there is some   $\mu$ such that $\tau=\sigma\conc \mu$.
When working with $2^{\leq t}$ for a fixed $t$, we will informally refer to elements of $\bn{t}$ as \emph{nodes} and elements of $\bl{t}$ as \emph{leaves}.

We first recall the discrete notion of trees for binary relations, as defined by Hodges \cite{hodges}.\footnote{Our formulation exchanges the places of the indexing parameters $\sigma\in\bn{t}$ and $\eta\in\bl{t}$ compared to the definition in \cite{hodges}. This switch is entirely cosmetic and has been done for  aesthetic reasons more relevant to our companion paper \cite{CT-QSAR}.}

\begin{definition}\label{def:hodgestree}
Fix $t\geq 1$ and $E\seq X\times Y$. A  \emph{$t$-tree} for $E$ consists of sequences $(x_{\sigma}: \sigma\in \bn{t})$ from $X$ and $(y_{\eta}: \eta\in \bl{t})$ from $Y$ such that for all $\sigma\in \bn{t}$ and $\eta,\eta'\in \bl{t}$, if  $\sigma\conc 0\iseq \eta$ and $\sigma\conc 1\iseq\eta'$ then $(x_\sigma,y_\eta)\not\in E$ and $(x_\sigma,y_{\eta'})\in E$. 
\end{definition}

The following is a generalization of this definition to functions.

\begin{definition}\label{def:eptree1}
Fix $t\geq 1$, $\delta>0$, and a function $f\colon X\times Y\to [0,1]$. A \emph{$(t,\delta)$-tree} for  $f$ consists of sequences
\[
(x_{\sigma}: \sigma\in \bn{t}) \text{ from }X,\quad (y_{\eta}: \eta\in \bl{t})\text{ from }Y, \quad \text{and}\quad (r_\sigma:\sigma\in \bn{t})\text{ from }[0,1]
\]
such that for all $\sigma\in \bn{t}$ and  $\eta,\eta'\in \bl{t}$, if  $\sigma\conc 0\iseq \eta$ and $\sigma\conc 1\iseq \eta'$ then $f(x_{\sigma},y_{\eta})\leq r_{\sigma}$ and $f(x_{\sigma},y_{\eta'})\geq r_{\sigma}+\delta$. 

Given such sequences, we refer to $t$ as the \emph{height} of the tree,  $\delta$ as the \emph{scale} of the tree, and $(r_\sigma:\sigma\in\bn{t})$ as the sequence of \emph{values} of the tree. The $x_{\sigma}$ are referred to as \emph{nodes} and the $y_{\eta}$ as \emph{leaves}.
\end{definition}

Note that a $t$-tree for a relation $E\seq X\times Y$ corresponds to a $(t,\delta)$-tree for the indicator function $\boldsymbol{1}_E\colon X\times Y\to [0,1]$, provided $\delta\leq 1$. Definition \ref{def:eptree1} also directly corresponds to sequential fat-shattering dimension (see Definition \ref{def:Ldim}) in analogy to the correspondence between Definition \ref{def:hodgestree} and Littlestone dimension (see \cite{Bhas,ChaFr2}).

\subsection{Ladders}\label{sec:ladders} We next turn to analogues of the order property for functions. In contrast to the case of trees,  several distinct variations of function-theoretic order properties have been considered in previous literature. To unify the presentation, we will phrase these definitions using analogues of Hodges' \cite{hodges} ``ladder" terminology from the discrete setting. Let us first again recall the definition in the discrete case.  

\begin{definition}\label{def:hodgesladder}
Fix $k\geq 1$ and $E\seq X\times Y$. A \emph{$k$-ladder} for $E$ consists of sequences $(x_1,\ldots,x_k)$ from $X$ and $(y_1,\ldots,y_k)$ from $Y$ such that for all $i,j\in[k]$, $(x_i,y_j)\in E$ if and only if $i\leq j$.
\end{definition}

Our main results will focus on three distinct generalizations of ladders to the setting of functions. 

\begin{definition}\label{def:ladders}
Fix $k\geq 1$, $\delta>0$, and a function $f\colon X\times Y\to [0,1]$.
\begin{enumerate}[$(1)$]
\item A \emph{$(k,\delta)$-ladder} for $f$ consists of sequences
\[
x_1,\ldots,x_k \in X,\quad y_1,\ldots,y_k\in Y,\quad\text{ and }\quad r_1,\ldots,r_k \in [0,1]
\]
such that for all $i,j \in [k]$, if $i \leq j$ then $f(x_i,y_j) \geq r_i+\delta$, and if $i>j$ then $f(x_i,y_j) \leq r_i$.

We refer to $k$ as the \emph{length} of the ladder, $\delta$ as the \emph{scale} of the ladder, and $r_1,\ldots,r_k$ as the \emph{values} of the ladder.

\item A \emph{uniform $(k,\delta)$-ladder} for $f$ is a $(k,\delta)$-ladder for which the values  $r_1,\ldots,r_k$ are all equal to a common value $r\in[0,1]$.

\item An \emph{$\e$-tight $(k,\delta)$-ladder} for $f$ is a uniform $(k,\delta)$-ladder for which there are intervals $I,J\seq[0,1]$, each of length at most $\epsilon$, such that for all $i,j\in[k]$, if $i\leq j$ then $f(x_i,y_j)\in I$, and if $i>j$ then $f(x_i,y_j)\in J$.
\end{enumerate} 
\end{definition}

When applied to the indicator function $\boldsymbol{1}_E$ of some relation $E\seq X\times Y$, these three configurations all coincide with a $k$-ladder for $E$, provided $\delta\leq 1$. 
 
In the learning theory literature, the ladder configuration is frequently referred to using the terminology of  ``thresholds", which originates in  \cite{ABMS,ALMM} in the discrete setting. The first notion of thresholds for real-valued functions was given by Jung, Kim, and Tewari \cite[Definition 7]{JuKiTe} using the stronger idea of ``tight thresholds". This  was later  refined by Daskalakis and Golowich  in \cite[Definition 8.1]{DaskGo}, which is the basis for our formulation of tight ladders. The definition of \emph{fat-threshold dimension} for a class of $[0,1]$-valued functions appears later in \cite{AADDF}, and essentially corresponds to our notion of uniform ladders. The non-uniform variation of ladders in Definition \ref{def:ladders}(1) has not, to our knowledge, appeared in the literature before. However, this version is the most useful from the perspective of our companion paper \cite{CT-QSAR}. Further details and discussion of these notions will be given in Appendix \ref{sec:learning}.

Finally, we define one more variation of a ladder, which corresponds to a common formulation of the order property in continuous logic  (see \cite[Definition 7.1]{BBHU}).

\begin{definition}\label{def:agnosticladder}
Given $k\geq 1$ and $\delta>0$, an \emph{agnostic $(k,\delta)$-ladder} for a function $f\colon X\times Y\to [0,1]$ consists of sequences $x_1,\ldots,x_k\in X$ and $y_1,\ldots,y_k\in Y$ such that for all distinct $i,j\in[k]$,
\[
|f(x_i,y_j)-f(x_j,y_i)|\geq \delta.
\]
\end{definition}

Agnostic ladders fit with the other ladders in  Definition \ref{def:ladders} via the easy observation that a uniform $(k,\delta)$-ladder is a special case of an agnostic $(k,\delta)$-ladder. Some further implications between these notions are given  in Subsection \ref{sec:ladderimps}.

\subsection{Prior work and open problems}\label{sec:prior}
In this subsection, we  recall the Shelah-Hodges correspondence between trees and ladders for binary relations. We then survey previous results in the literature related to the function-theoretic extension of this correspondence. 

We first set some terminology. Given $E\seq X\times Y$, we say that $E$ \emph{admits} a $t$-tree (resp., $k$-ladder) if a $t$-tree (resp., $k$-ladder) for $E$ exists. Otherwise, we say $E$ \emph{omits} $t$-trees (resp., $k$-ladders). We will use the analogous terminology for admitting/omitting $(t,\delta)$-trees and (uniform/$\e$-tight) $(k,\delta)$-ladders for functions $f\colon X\times Y\to [0,1]$. 

We now state the Shelah-Hodges correspondence, with bounds from \cite{hodges}. 

\begin{theorem}[Shelah \cite{shelah}, Hodges \cite{hodges}]\label{thm:hodges}  Fix $E\seq X\times Y$.
\begin{enumerate}[$(1)$]
\item Given $t\geq 1$, if $E$ admits a $2^t$-ladder, then $E$ admits a $t$-tree.  
\item Given $k\geq 1$, if $E$ admits a $(2^{k+1}-2)$-tree, then $E$ admits a $k$-ladder.
\end{enumerate}
\end{theorem}

Hodges' proof of part (1) is a natural finitization of Shelah's infinitary proof in \cite{shelah}, which itself is fairly straightforward. On the other hand, Hodges' proof of part (2) is a more intricate and novel argument involving a Ramsey-theoretic coloring result for trees (see Lemma \ref{lem:ramsey} and the surrounding discussion).

As it will be relevant to our results, we note the following consequence of Theorem \ref{thm:hodges}.  In particular, Theorem \ref{thm:hodges}(2)   implies that for any $t\geq 1$,  if $E\seq X\times Y$ admits a $(2^{2^t+1}-2)$-tree, then $E$ admits a $2^t$-ladder which, after reversing the order of indices, becomes  a $2^t$-ladder in the ``dual" relation $E^{opp}\coloneqq\{(y,x)\in Y\times X:(x,y)\in E\}$.  Hence $E^{opp}$ admits a $t$-tree by Theorem \ref{thm:hodges}(1). In  learning theory, this translates to the result that if the Littlestone dimension of a set system is $t$, then the dual Littlestone dimension is at most $2^{2^{t+1}+1}-3$.

We now move to the setting of functions. In this case, the easier direction of Theorem \ref{thm:hodges} (namely, part (1)) generalizes in a clear way to yield the extraction of a $(t,\delta)$-tree from the ladders defined in Subsection \ref{sec:ladders}. The following are three results of this flavor from the learning theory literature (translated to our terminology and paraphrased).
\begin{enumerate}[$(1)$]
\item Daskalakis \& Golowich \cite[Lemma 8.1]{DaskGo}: An $\e$-tight $(2^t,\delta)$-ladder yields a $(t,\delta)$-tree.
\item Anderson \& Benedikt \cite[Theorem 45 (first bullet)]{AndBen}: A slightly weaker version of a uniform $(2^{t+1}-1,\delta)$-ladder yields a $(t,\delta)$-tree. \item Assos, Attias, Dagan, Daskalakis, \& Fishelson \cite[Lemma 25]{AADDF}: A uniform $(2^t+1,\delta)$-ladder yields a $(t,\delta)$-tree.
\end{enumerate}
In Subsection \ref{sec:easyhodges} of the appendix, we will use a similar argument to  show that  a  $(2^t,\delta)$-ladder yields a $(t,\delta)$-tree. 

The real focus of the paper is on the other direction: obtaining a ladder from a tree.  We now state two results along these lines from the previous literature.

\begin{theorem}[Anderson \& Benedikt \cite{AndBen}]\label{thm:AB}
Given $k \geq 1$ and $\delta,\epsilon > 0$, there exists some 
$t \leq O((1/\epsilon)^k)$ such that if
$f : X \times Y \to [0,1]$ admits a $(t,2\delta+\epsilon)$-tree then $f$
admits an agnostic $(k,\delta)$-ladder.
\end{theorem}

\begin{theorem}[Daskalakis \& Golowich \cite{DaskGo}]\label{thm:DG}
Given $k \geq 1$ and $\delta,\epsilon > 0$, there exists some
$t \leq 2^{(1/\epsilon)^{O(k/\epsilon)}}$ such that if
$f : X \times Y \to [0,1]$ admits a $(t,2\delta+\epsilon)$-tree, then $f$
admits an $\epsilon$-tight $(k,\delta)$-ladder.
\end{theorem}

For better comparison to our main results, Theorems \ref{thm:AB} and \ref{thm:DG} have been stated in a form slightly sharper  than  how they appear in the cited sources. In Sections \ref{sec:AB} and \ref{sec:DG}, we will explain how these statements follow directly from their respective proofs. 

We now describe two open problems from \cite{DaskGo} related to Theorem \ref{thm:DG}.

\begin{problem}\label{prob:open1}
When combined with  \cite[Lemma 8.1]{DaskGo} (one of the analogues of Theorem \ref{thm:hodges}(1) mentioned above), Theorem \ref{thm:DG} implies that if a function class $\cF\seq [0,1]^X$ has $\delta$-sequential fat-shattering dimension $t$ then, for any $\e>0$, the dual class has $(2\delta+\e)$-sequential fat-shattering dimension at most $2^{(1/\e)^{O(2^t/\e)}}$. 
Using similar ideas, one can derive the same conclusion from Theorem \ref{thm:AB} (see Subsection \ref{rem:finalremarks}). Note that this  is worse than the double-exponential bound in the analogous discrete result for Littlestone dimension mentioned above after Theorem \ref{thm:hodges}. The reconciliation of this gap is also left as an open problem in \cite{DaskGo}.
\end{problem}

\begin{problem}\label{prob:open2}
Prior to \cite{DaskGo}, a result very close to Theorem \ref{thm:DG}, but with the bound $(1/\e)^{O(k/\e^2)}$, was claimed by Jung, Kim, and Tewari \cite{JuKiTe}.
However, Daskalakis and Golowich identified a flaw in
the argument (see the discussion following \cite[Lemma 8.2]{DaskGo}). The gap is filled with Theorem \ref{thm:DG} at the cost of a weaker bound, and the question of recovering the stronger bound claimed in \cite{JuKiTe} is left open in \cite{DaskGo}.
\end{problem}

Both of these problems will be addressed by the main results of this paper, which we describe in the next subsection. In particular, our first main result, Theorem \ref{thm:hodgesfn1}, will lead to a double exponential bound for dual sequential fat-shattering dimension (see Corollary \ref{cor:dualbound}), resolving Problem \ref{prob:open1}. The argument used to prove Corollary \ref{cor:dualbound} will also imply a stronger version of Theorem \ref{thm:AB} with improved bounds (see Corollary \ref{cor:hodgesfn1}). Our second main result, Theorem \ref{thm:tight}, gives  Theorem \ref{thm:DG} with the stronger bound originally claimed in \cite{JuKiTe}, resolving Problem \ref{prob:open2}.

\subsection{Main results}\label{sec:mainresults}

Our first main result is a function-theoretic analogue of Theorem \ref{thm:hodges}$(2)$ for non-uniform ladders with significantly improved bounds compared to what one can obtain from either Theorem \ref{thm:AB} or Theorem \ref{thm:DG}.\footnote{In Subsection \ref{rem:finalremarks}, we explain the precise form of this result one can deduce from these theorems.} 

\begin{theorem}\label{thm:hodgesfn1}
Given $k \geq 1$ and $\delta > 0$, if $f : X \times Y \to [0,1]$ admits a
$(\binom{2k}{k}-1,2\delta)$-tree, then $f$ admits a $(k,\delta)$-ladder.
\end{theorem}

The proof of Theorem \ref{thm:hodgesfn1} is given in Section \ref{sec:hodgesfn1}. We follow an induction scheme inspired by Hodges' argument in the discrete case. In particular, we will use a Ramsey-theoretic result for colorings of trees (see Lemma \ref{lem:ramsey}), which also lies at the heart of Hodges' proof. However, there are additional subtleties due to asymmetries not present in the discrete setting. 

The bound $\binom{2k}{k}-1$ in Theorem \ref{thm:hodgesfn1} grows on the order of $4^k/\sqrt{\pi k}$, which could potentially be further improved (see Remark \ref{rem:2kkbound}). Note also that Theorem \ref{thm:hodgesfn1} contains the same scale change from $\delta$ to $2\delta$ present in Theorems \ref{thm:AB} and \ref{thm:DG}. Proposition \ref{prop:GPT1} will show  that this change is necessary for  Theorem \ref{thm:hodgesfn1} even at the qualitative level.

By combining Theorem \ref{thm:hodgesfn1} with a Ramsey-theoretic result on trees (Proposition \ref{prop:unitreeimp}), we will obtain the following variation for uniform ladders.

\begin{corollary}\label{cor:hodgesfn1}
Given $k \geq 1$ and $\delta,\epsilon > 0$, there exists some
$t \leq \lceil\epsilon^{-1}\rceil 4^k$ such that if
$f : X \times Y \to [0,1]$ admits a $(t,2\delta+\epsilon)$-tree, then $f$
admits a uniform $(k,\delta)$-ladder.
\end{corollary}

Recall that a uniform $(k,\delta)$-ladder is, in particular, an agnostic $(k,\delta)$-ladder. Therefore Corollary \ref{cor:hodgesfn1} provides a stronger version of Theorem \ref{thm:AB} with improved bounds.

Together with the function-theoretic version of extracting trees from ladders (e.g., Theorem \ref{thm:optree} or \cite[Lemma 25]{AADDF}),  Corollary \ref{cor:hodgesfn1} immediately implies the following conclusion on transferring trees from a function $f\colon X\times Y\to [0,1]$ to its dual $f^{opp}\colon Y\times X\to [0,1]$ (which maps $(y,x)$ to $f(x,y)$). 

\begin{corollary}\label{cor:dualbound}
Given $t \geq 1$ and $\delta,\epsilon > 0$, there is some
$t^* \leq \lceil\epsilon^{-1}\rceil 2^{2^{t+1}}$ such that if
$f : X \times Y \to [0,1]$ admits a $(t^*,2\delta+\epsilon)$-tree, then
$f^{opp}$ admits a $(t,\delta)$-tree.
\end{corollary}

When translated to the language of learning theory, Corollary~\ref{cor:dualbound}
says that if a function class $\mathcal{F} \subseteq [0,1]^X$ has sequential
$\delta$-fat-shattering dimension $t$, then for all
$\epsilon > 0$, the dual class has sequential $(2\delta+\e)$-fat-shattering dimension less than
$\lceil\epsilon^{-1}\rceil 2^{2^{t+2}}$.\footnote{This bound can be sharpened slightly, as noted after the proof of Corollary \ref{cor:dualbound} in Section \ref{sec:hodgesfn1}.} This double-exponential bound matches the corresponding result in the discrete case for Littlestone dimension, resolving Problem \ref{prob:open1}.

Our second main result directly addresses Problem \ref{prob:open2}. In particular, we give a new proof of Theorem \ref{thm:DG} with the stronger bound originally claimed in \cite{JuKiTe}.

\begin{theorem}\label{thm:tight}
Given $k \geq 1$ and $\delta,\epsilon > 0$, there is some
$t \leq (1/\epsilon)^{O(k/\epsilon^2)}$ such that if
$f : X \times Y \to [0,1]$ admits a $(t,2\delta+\epsilon)$-tree, then $f$
admits an $\epsilon$-tight uniform $(k,\delta)$-ladder.
\end{theorem}

The proof of Theorem \ref{thm:tight} is given in Section \ref{sec:tight}, and is based on an extensive elaboration of the proof of our first main result (Theorem~\ref{thm:hodgesfn1}). To explain how our proof obtains the stronger bound, we first discuss the rough idea of  Daskalakis and Golowich's proof of Theorem \ref{thm:DG}. In particular, their argument first extracts from a tree a configuration resembling an agnostic ladder
(Definition~\ref{def:agnosticladder}), but with  additional tightness on ``half" of the ladder.
This step uses the same tree coloring result mentioned above
(Lemma~\ref{lem:ramsey}) and yields a bound on the order of $(1/\epsilon)^k$.
From there, a multi-colored Ramsey argument is used to turn this configuration
into a uniform tight ladder, which results in a bound of the form
$2^{(1/\epsilon)^{O(k/\epsilon)}}$. In contrast, our proof of
Theorem~\ref{thm:tight} will incorporate a companion ``leaf coloring'' result
(Lemma~\ref{lem:ramseyleaves}), which will be used in parallel with
Lemma~\ref{lem:ramsey} in order to ensure full tightness on the  ladder
extracted from a tree. This allows us to  avoid multi-colored Ramsey numbers altogether.

\subsection{Notation, terminology, and conventions}\label{sec:notation}
\begin{enumerate}[$(1)$]
\setlength{\itemsep}{5pt}
\item Several results will include asymptotic notation such as $O(1/\e)$ where $\epsilon>0$ is a fixed real parameter.   In order for this  to make sense, we tacitly assume $\e<\frac{1}{2}$ in these statements. 
\item Throughout the paper, $\log$ denotes  $\log_2$.
\item Given an integer $n\geq 1$,  set $[n]:=\{1,\ldots, n\}$.  
\item A \emph{cover} of a set $X$ is a collection $X_1,\ldots, X_n\subseteq X$ such that $X=X_1\cup \ldots \cup X_n$.

 \item \label{not:opp} Given $f\colon X\times Y\rightarrow [0,1]$, define $f^{opp}\colon Y\times X\to [0,1]$ so that $f^{opp}(y,x)=f(x,y)$.

\item Given a function $f\colon X\times Y\to [0,1]$, we say that $f$ \emph{admits} a $(t,\delta)$-tree if this configuration exists for $f$ as described in Definition \ref{def:eptree1}. Otherwise, we say $f$ \emph{omits} $(t,\delta)$-trees. We use the analogous terminology for $(k,\delta)$-ladders, uniform $(k,\delta)$-ladders, $\e$-tight $(k,\delta)$-ladders, and agnostic $(k,\delta)$-ladders (described in Definitions \ref{def:ladders} and \ref{def:agnosticladder}).
\end{enumerate}

\subsection{Outline of the paper}

\begin{enumerate}[\hspace{5pt}$\ast$]
\setlength{\itemsep}{5pt}
\item Section \ref{sec:treetools}: We define tree embeddings and develop some basic tools around this notion, including the two Ramsey coloring lemmas for trees mentioned above (Lemmas \ref{lem:ramsey} and \ref{lem:ramseyleaves}). We also define a uniform analogue of trees and prove a result on extracting uniform trees from trees (Proposition \ref{prop:unitreeimp}). 
\item Section \ref{sec:hodgesfn1}: We prove Theorem \ref{thm:hodgesfn1}, Corollary \ref{cor:hodgesfn1}, and Corollary \ref{cor:dualbound}. 
\item Section \ref{sec:tight}: We prove Theorem \ref{thm:tight}. 
\item Section \ref{sec:agnostic}: We define an agnostic analogue of trees and prove a result on extracting a uniform tree from an agnostic tree (Theorem \ref{thm:OPsharp}). This section is not needed for any of our main results, but it is relevant to the theme of the paper.
\item Appendix \ref{sec:moreladders}: We prove some further implications between ladders, uniform ladders, and tight ladders. We then prove  our version of the function-theoretic analogue of Theorem \ref{thm:hodges}(1), which extracts a tree from a non-uniform ladder (see Theorem \ref{thm:optree}).

\item Appendix \ref{sec:learning}: We translate between statistical learning theory and our combinatorial setting of binary functions. We then justify our formulations of Theorems \ref{thm:AB} and \ref{thm:DG}.

\item Appendix \ref{ss:2rankcont}: We briefly discuss the earlier foundations of the function-theoretic order-tree correspondence in continuous logic. 
\end{enumerate}

\section{Tools for trees}\label{sec:treetools}

\subsection{Tree embeddings}

An important set of tools for us will be maps between trees  which respect the underlying tree structure.  We make this precise with the following definition.  

\begin{definition}\label{def:TE} Fix integers $s,t\geq 0$.
\begin{enumerate}[$(1)$]
\item  A function $\phi\colon \bne{s}\rightarrow \bne{t}$ is a \emph{tree embedding} if for all $\sigma,\tau\in \bne{s}$ and $\alpha\in \{0,1\}$, if $\sigma\conc \alpha\iseq \tau$ then $\phi(\sigma)\conc \alpha\iseq \phi(\tau)$. 
\item A function $\phi\colon \bne{s}\to\bne{t}$ is a \emph{proper tree embedding} if it is a tree embedding and, moreover, $\phi(\bl{s})\seq \bl{t}$. 
\end{enumerate}
\end{definition}

The idea of a proper tree embedding is that it is a tree embedding sending nodes to nodes and leaves to leaves. This is clarified in part $(c)$ of the next proposition, which states several basic properties of tree embeddings.

\begin{proposition}\label{prop:leavesemb}
$~$
\begin{enumerate}[$(a)$]
\item Given $s,t\geq 1$, if $\phi\colon\bn{s}\rightarrow\bn{t}$ is a tree embedding then, for any $\sigma,\tau\in \bn{s}$, $\sigma\iseq \tau$ if and only if $\phi(\sigma)\iseq\phi(\tau)$. In particular, $\phi$ is injective.
\item Given $s,t,u\geq 1$, if $\phi\colon \bn{s}\to \bn{t}$ and $\psi\colon \bn{t}\to\bn{u}$ are tree embeddings then $\psi\circ\phi$ is a tree embedding from $\bn{s}$ to $\bn{u}$. Moreover, if $\phi$ and $\psi$ are both proper, then so is $\psi\circ\phi$. 
\item Given $s,t\geq 1$, if $\phi\colon\bne{s}\to \bne{t}$ is a proper tree embedding then $\phi(\bn{s})\seq \bn{t}$. 
\item Given $s,t\geq 1$, any tree embedding $\phi\colon \bn{s}\to \bn{t}$ can be extended to a proper tree embedding $\phi^*\colon\bne{s}\to\bne{t}$.
\end{enumerate}
\end{proposition}
\begin{proof}
Part $(a)$ is left as an exercise.\footnote{This statement will not be needed for our results, and is only included for later comparison to \cite{hodges}.} Parts $(b)$ and $(c)$ are straightforward. 

Part $(d)$. 
Define $\phi^*\colon \bne{s}\to \bne{t}$ such that $\phi^*$ extends  $\phi$ and, for each $\sigma\in \bl{s-1}$ and $\alpha\in\{0,1\}$, $\phi^*(\sigma\conc\alpha)$ is some extension of $\phi(\sigma)\conc\alpha$ of length $t$ (this exists since $\phi(\sigma)\in\bn{t}$,  hence $\phi(\sigma)\conc\alpha$ has length at most $t$). Then $\phi^*$ is a proper tree embedding extending $\phi$.
\end{proof}

 We note that our notion of tree embedding is stronger than Hodges' \cite{hodges} notion of a \emph{tree map}, which only requires part $(a)$ of the previous proposition.   This difference is largely irrelevant in the discrete case because $t$-trees for binary relations (Definition \ref{def:hodgestree}) exhibit a certain symmetry not present in  $(t,\delta)$-trees for functions (Definition \ref{def:eptree1}). A variation of Definition \ref{def:eptree1} with this additional symmetry will be studied later in Section \ref{sec:agnostic}.

Finally, we show that proper tree embeddings respect trees admitted by bipartite functions.

\begin{proposition}\label{prop:embedtree}
Fix $s,t\geq 1$ and $\delta>0$. 
Assume $f\colon X\times Y\to [0,1]$ admits a $(t,\delta)$-tree consisting of sequences  $(x_\sigma:\sigma\in \bn{t})$ from $X$, $(y_\eta:\eta\in\bl{t})$ from $Y$, and $(r_\sigma:\sigma\in\bn{t})$ from $[0,1]$. Suppose $\phi\colon\bne{s}\to \bne{t}$ is a proper tree embedding. Then $(x_{\phi(\sigma)}:\sigma\in\bn{s})$, $(y_{\phi(\eta)}:\eta\in \bl{s})$, and $(r_{\phi(\sigma)}:\sigma\in\bn{s})$ form an $(s,\delta)$-tree for $f$.
\end{proposition}
\begin{proof}
First  note that since $\phi$ is proper, we have $\phi(\bn{s})\seq \bn{t}$ and $\phi(\bl{s})\seq\bl{t}$. Hence the sequences in the conclusion are well-defined. The verification that these sequences form an $(s,\delta)$-tree is immediate from the fact that  for any $\sigma\in\bn{s}$ and $\eta,\eta'\in \bl{s}$, if $\sigma\conc 0\iseq \eta$ and $\sigma\conc 1\iseq \eta'$ then, since $\phi$ is a tree embedding, $\phi(\sigma)\conc 0\iseq \phi(\eta)$ and $\phi(\sigma)\conc 1\iseq \phi(\eta')$. 
\end{proof}

\subsection{Ramsey lemmas for trees}\label{sec:ramsey}

Next we formulate two Ramsey-type results on tree colorings. The first is essentially the same as a lemma proved by Hodges \cite{hodges}, which gives the analogous result for the weaker notion of tree maps.  Our version follows from Hodges' version by composing with an appropriate level-preserving automorphism of a tree. However, for the sake of completeness, we will provide a detailed proof. We first define two operations on functions between trees. 

\begin{definition}$~$ Let $t,s\geq 1$ be integers. 
\begin{enumerate}[$(i)$]
\item Given $\phi\colon \bn{s}\to\bn{t}$ and $u\in\{0,1\}$, define $\phi^u\colon \bn{s}\to \bne{t}$ by setting $\phi^u(\sigma)=u\conc\phi(\sigma)$ for all $\sigma\in \bn{s}$.
\item Given $\phi_0,\phi_1\colon \bn{s}\to\bn{t}$, define $[\phi_0,\phi_1]\colon \bne{s}\to\bne{t}$ by setting $[\phi_0,\phi_1](\emp)=\emp$, and $[\phi_0,\phi_1](u\conc\sigma)=u\conc\phi_u(\sigma)$ for each $u\in\{0,1\}$ and $\sigma\in \bn{s}$.
\end{enumerate}
\end{definition}

These operations will be applied to tree embeddings. In this case,  $\phi^u$ shifts the embedded copy of $\bn{s}$ down one level and  left or right depending on $u$. Likewise, $[\phi_0,\phi_1]$ embeds $\bne{s}$ by shifting down the two copies of $\bn{s}$ embedded  by $\phi_0$ and $\phi_1$ and joining them with root $\emp$. With this in mind, the following is a straightforward exercise.

\begin{remark}\label{rem:treeops}$~$
\begin{enumerate}[$(a)$]
\item If $\phi\colon \bn{s}\to\bn{t}$ is a tree embedding, then so are $\phi^0$ and $\phi^1$.
\item If $\phi_0,\phi_1\colon\bn{s}\to\bn{t}$ are tree embeddings, then so is $[\phi_0,\phi_1]$.
\end{enumerate}
\end{remark}

We now prove the first Ramsey result.

\begin{lemma}[Ramsey for tree embeddings]\label{lem:ramsey}
Fix integers $m,t\geq 1$ and $t_1,\ldots, t_m\geq 1$ such that $t=t_1+\ldots +t_m-m+1$.  Then for any cover
$$
\bn{t}=C_1\cup \ldots \cup C_m,
$$
there is some $1\leq i\leq m$ and a tree embedding $\phi\colon \bn{t_i}\rightarrow \bn{t}$ with  image contained in $C_i$.
\end{lemma}
\begin{proof}
We proceed by induction on $t\geq 1$. Note that the base case $t=1$ holds trivially.

Assume now $t>1$ and suppose the result holds for $t-1$.
Fix $m\geq 1$ and $t_1,\ldots, t_m\geq 1$ satisfying $t=t_1+\ldots +t_m-m+1$. Let $\bn{t}=C_1\cup \ldots \cup C_m$ be a cover. This yields two covers of $\bn{t-1}$, namely,
\begin{equation}\label{part}
\bn{t-1}=C_1^0\cup \ldots \cup C_m^0=C_1^1\cup \ldots \cup C_m^1,
\end{equation}
where, for each $i\in[m]$ and $u\in\{0,1\}$, we define
$$
C_i^u=\{\sigma\in \bn{t-1}: u\conc \sigma\in C_i\}.
$$

By relabeling if necessary, we may assume $\emp\in C_m$.
Note that if $t_m=1$ then the map from $\bn{t_m}$ to $\bn{t}$ sending $\emp$ to $\emp$ is a tree embedding with image contained in $C_m$, as desired. So we may assume $t_m\geq 2$.
Since $t_1+\ldots+t_{m-1}+(t_m-1)-m+1=t-1$, we can apply our induction hypothesis  to the covers in (\ref{part}) and conclude that one of the following holds.
\begin{enumerate}[$(a)$]
\item For some $u\in \{0,1\}$ and $i\in [m-1]$, there is a tree embedding $\psi\colon \bn{t_i}\rightarrow \bn{t-1}$ with $\Im(\psi)\seq C_i^u$.
\item For each $u\in\{0,1\}$, there is a tree embedding $\phi_u\colon \bn{t_m-1}\rightarrow \bn{t-1}$ with $\Im(\phi_u)\seq C_m^u$. 
\end{enumerate}

Suppose first  $(a)$ holds. By Remark \ref{rem:treeops}$(a)$, we have a tree embedding $\phi\coloneqq\psi^u\colon \bn{t_i}\to \bn{t}$. Moreover, since $\Im(\psi)\seq C^u_i$, we have $\Im(\phi)\seq C_i$.

Assume now $(b)$ holds.  By Remark \ref{rem:treeops}$(b)$, the map $\phi\coloneqq [\phi_0,\phi_1]\colon \bn{t_m}\to\bn{t}$ is a tree embedding. Moreover, since $\emp\in C_m$ and $\Im(\phi_u)\seq C^u_m$ for each $u\in\{0,1\}$, we have $\Im(\phi)\seq C_m$.
\end{proof}

 We will also need a variation of Lemma \ref{lem:ramsey}  suitable for \emph{proper} tree embeddings.

\begin{lemma}[Ramsey for proper tree embeddings]\label{lem:ramseyleaves}
Fix integers $m\geq 1$ and $t,t_1,\ldots, t_m\geq 0$ such that $t=t_1+\ldots +t_m$.  Then for any cover
$$
\bl{t}=C_1\cup \ldots \cup C_m,
$$
there is some $i\in[m]$ and a proper tree embedding $\phi:\bne{t_i}\rightarrow \bne{t}$ such that  $\phi(\bl{t_i})\subseteq C_i$.
\end{lemma}
\begin{proof}
We proceed by induction on $t\geq 0$, with $t=0$ a trivial base case. Assume  $t>0$, and suppose the result holds for $t-1$. Fix integers $m\geq 1$ and $t_1,\ldots, t_m\geq 0$ satisfying $t=t_1+\ldots +t_m$, and a cover $\bl{t}=C_1\cup \ldots \cup C_m$. As in the proof of Lemma \ref{lem:ramsey} above, this   yields the covers
\begin{align}\label{part2}
\bl{t-1}=C_1^0\cup \ldots \cup C_m^0=C_1^1\cup \ldots \cup C_m^1, 
\end{align}
where $C_i^u=\{\eta\in \bl{t-1}: u\conc \eta\in C_i\}$.

By relabeling if necessary, we may assume $t_m\geq 1$ (recall $t>0$). Since $t_1+\ldots +t_{m-1}+(t_m-1)=t-1$, we can apply our induction hypothesis  to the covers in (\ref{part2}), and conclude that one of the following holds.
\begin{enumerate}[$(a)$]
\item For some $i\in[m-1]$ and $u\in \{0,1\}$, there is a proper tree embedding $\psi\colon \bne{t_i}\rightarrow \bn{t}$ with $\psi(\bl{t_i})\seq C_i^u$.
\item For all $u\in\{0,1\}$, there is a proper tree embedding $\phi_u\colon \bn{t_m}\to \bn{t}$ with $\phi_u(\bl{t_m-1})\seq C^u_m$. 
\end{enumerate}

Suppose first $(a)$ holds. By Remark \ref{rem:treeops}$(a)$, the map $\phi\coloneqq\psi^u\colon \bne{t_i}\to \bne{t}$ is a tree embedding. Moreover, since $\psi(\bl{t_i})\seq C^u_i$, we have $\phi(\bl{t_i})\seq C_i$ (hence $\phi$ is proper).

Assume now $(b)$ holds.  By Remark \ref{rem:treeops}$(b)$, the map $\phi\coloneqq [\phi_0,\phi_1]\colon \bne{t_m}\to\bne{t}$ is a tree embedding. Moreover, since $\phi_u(\bl{t_m-1})\seq C^u_m$ for each $u\in\{0,1\}$, we have $\phi(\bl{t_m})\seq C_m$.
\end{proof}

\subsection{Uniform trees}

We now define a uniform variation of $(t,\delta)$-trees (analogous to the uniform variation of ladders in Definition \ref{def:ladders}(2)).

\begin{definition}\label{def:uniformtree}
Given $t\geq 1$ and $\delta>0$, a \emph{uniform $(t,\delta)$-tree for $f\colon X\times Y\to [0,1]$} is a $(t,\delta)$-tree in which the values $(r_\sigma:\sigma\in \bn{t})$ are all equal to some  $r\in[0,1]$. When this exists, we say that $f$ \emph{admits a uniform $(t,\delta)$-tree}, and we call $r$ the \emph{value} for the tree.
\end{definition}

Clearly a uniform $(t,\delta)$-tree for a function $f$ is a $(t,\delta)$-tree in the sense of Definition \ref{def:eptree1}. The following approximate converse implication will be an important ingredient in our later results.

\begin{proposition}\label{prop:unitreeimp}
Given $t\geq 1$ and $\delta,\e>0$, if $f\colon X\times Y\to [0,1]$ admits a $(t',\delta+\e)$-tree, where $t'=\lceil \e\inv\rceil (t-1)+1$, then $f$ admits a uniform $(t,\delta)$-tree.
\end{proposition}
\begin{proof}
Set $T=\lceil \e\inv\rceil (t-1)+1$. Assume $f$ admits a  $(T,\delta+\e)$-tree with nodes $(x_{\sigma}: \sigma\in \bn{T})$, leaves $(y_{\eta}: \eta\in \bl{T})$ and values $(r_{\sigma}: \sigma\in \bn{T})$.  Set $K=\lceil \e\inv\rceil$, and for each $i\in[K]$, define
\[
C_{i}=\left\{\sigma \in \bn{T}:  \frac{i-1}{K}\leq r_\sigma\leq  \frac{i}{K}\right\}.
\]
This forms a cover $\bn{T}=C_1\cup \ldots \cup C_{K}$. Since $T=K(t-1)+1$, Lemma \ref{lem:ramsey}, applied with $m=K$ and $t_1=\ldots=t_K=t$, yields some $i\in[K]$ and a tree embedding $\phi\colon \bn{t}\rightarrow \bn{T}$ with image contained in $C_i$. By Proposition \ref{prop:leavesemb}$(d)$, we can extend $\phi$ to a proper tree embedding $\phi^*\colon \bne{t}\to\bne{T}$. By Proposition \ref{prop:embedtree}, the sequences $\bar{x}=(x_{\phi^*(\sigma)}:\sigma\in\bn{t})$, $\bar{y}=(y_{\phi^*(\eta)}:\eta\in \bl{t})$, and $\bar{r}=(r_{\phi^*(\sigma)}:\sigma\in\bn{t})$ form a $(t,\delta+\e)$-tree for $f$. Since $\Im(\phi^*|_{\bn{t}})=\Im(\phi)\seq C_i$, each value $r_{\phi^*(\sigma)}$ lies in the interval $[\frac{i-1}{K},\frac{i}{K}]$, which has length $\frac{1}{K}\leq \e$. Thus for any $\sigma\in\bn{t}$, $r_{\phi^*(\sigma)}\leq \frac{i}{K}$ and $r_{\phi^*(\sigma)}+\delta+\e\geq \frac{i}{K}+\delta$.
It follows that $\bar{x}$ and $\bar{y}$ form a uniform $(t,\delta)$-tree for $f$ with value $\frac{i}{K}$. 
\end{proof}

\section{Efficient extraction of ladders from trees}\label{sec:hodgesfn1}

We now prove Theorem \ref{thm:hodgesfn1}, which  says that if $f\colon X\times Y\to [0,1]$ admits a $(\binom{2k}{k}-1,2\delta)$-tree, then $f$ admits a $(k,\delta)$-ladder.

\begin{proof}[\textnormal{\textbf{Proof of Theorem \ref{thm:hodgesfn1}}}]
Fix $\delta>0$ and $f\colon X\times Y\to [0,1]$. 
For $k,\ell,t\geq 1$, let $P(k,\ell,t)$ be the following statement:\medskip

\noindent\textit{Statement of $P(k,\ell,t)$.}
If $(x_{\sigma}: \sigma \in \bn{t})$ and $(y_{\eta}: \eta \in \bl{t})$ form  a  $(t,2\delta)$-tree for $f$ with values $(r_\sigma:\sigma\in\bn{t})$,  then one of the following holds.
\begin{enumerate}
\item[\textnormal{(I)}] There are maps $\alpha\colon[k]\to \bn{t}$ and $\beta\colon [k]\to \bl{t}$ such that  for all $i,j\in [k]$, if $i\leq j$ then $f(x_{\alpha(j)},y_{\beta(i)})\geq r_{\alpha(j)}+\delta$,
and if $ i<j$ then $f(x_{\alpha(i)},y_{\beta(j)})\leq r_{\alpha(i)}$.
\item[\textnormal{(II)}] There are maps $\alpha\colon[\ell]\to \bn{t}$ and $\beta\colon [\ell]\to \bl{t}$ such that for all $i,j\in [\ell]$, if $i\leq j$ then $f(x_{\alpha(j)},y_{\beta(i)})< r_{\alpha(j)}+\delta$, and if $i<j$ then $f(x_{\alpha(i)},y_{\beta(j)})\geq r_{\alpha(i)}+2\delta$.
\end{enumerate}\medskip

Using Proposition \ref{prop:embedtree}, it is easy to see that if $P(k,\ell,t)$ holds for some $t\geq 1$, then $P(k,\ell,t')$ holds for any $t'\geq t$.
Given $k,\ell\geq 1$, define $T(k,\ell)$ to be the least integer $t$ (if it exists) such that $P(k,\ell,t)$ holds. The next claim will show that $T(k,\ell)$ exists for all $k$ and $\ell$.  \medskip

\noindent\textit{Claim 1.} 
\begin{enumerate}[$(a)$]
\item For all $k,\ell\geq 1$,  $T(k,1)=T(1,\ell)=1$.
\item For all $k,\ell\geq 2$, if $T(k-1,\ell)$ and $T(k,\ell-1)$ exist then $T(k,\ell)$ exists and 
\[
T(k,\ell)\leq T(k-1,\ell)+T(k,\ell-1)+1.
\]
\end{enumerate}
\noindent\textit{Proof.} Part $(a)$. Suppose $x_{\semp}$ and $(y_0,y_1)$ form  a  $(1,2\delta)$-tree for $f$ with value $r_{\semp}$. Then $f(x_{\semp},y_0)\leq r_{\semp}<r_{\semp}+\delta$,  and so $P(k,1,1)$ holds witnessed by (II) with $\alpha(1)=\emp$ and $\beta(1)=0$. Also, $f(x_{\semp},y_1)\geq r_{\semp}+2\delta\geq r_{\semp}+\delta$, 
and so $P(1,\ell,1)$ holds witnessed by (I) with $\alpha(1)=\emp$ and $\beta(1)=1$.\medskip

Part $(b)$. Assume $T(k-1,\ell)$ and $T(k,\ell-1)$ exist, and set $t=T(k-1,\ell)+T(k,\ell-1)+1$. We will show that $P(k,\ell,t)$ holds. Toward this end, suppose $\bar{x}=(x_{\sigma}: \sigma\in \bn{t})$ and $\bar{y}=(y_{\tau}: \tau\in \bl{t})$ form a $(t,2\delta)$-tree for $f$ with values $\bar{r}=(r_\sigma:\sigma\in\bn{t})$. Fix any $\eta_*\in \bl{t}$. Define
\[
P_0:=\{ \sigma\in \bn{t}: f(x_{\sigma},y_{\eta_*})\geq r_\sigma+\delta\} \quad\text{and}\quad
P_1:=\{\sigma\in \bn{t}: f(x_{\sigma},y_{\eta_*})<r_\sigma+\delta\}.
\]
Note $P_0\cup P_1=\bn{t}$. Set $t_0=T(k-1,\ell)+1$ and $t_1=T(k,\ell-1)+1$. Then $t=t_0+t_1-1$. Thus Lemma \ref{lem:ramsey} implies that for some $u\in\{0,1\}$, there is a tree embedding $\phi:\bn{t_u}\rightarrow \bn{t}$ with $\phi(\bn{t_u})\seq P_u$. By Proposition \ref{prop:leavesemb}$(d)$, we can extend $\phi$ to a proper tree embedding $\phi^*\colon \bne{t_u}\to\bne{t}$. Define $\psi\colon \bn{t_u}\to\bne{t}$ such that $\psi(\sigma)=\phi^*(u\conc\sigma)$. Then $\psi$ is a proper tree embedding by Proposition \ref{prop:leavesemb}$(b)$ and the fact that $\sigma\mapsto u\conc\sigma$ is a proper tree embedding from $\bn{t_u}$ to $\bne{t_u}$. Therefore, by Proposition \ref{prop:embedtree}, $\bar{x}'=(x_{\psi(\sigma)}:\sigma\in\bn{t_u-1})$ and $\bar{y}'=(y_{\psi(\eta)}:\eta\in\bl{t_u-1})$ form a $(t_u-1,2\delta)$-tree for $f$ with values $\bar{r}'=(r_{\psi(\sigma)}:\sigma\in\bn{t_u-1})$. We now analyze two cases depending on the value of $u$. \medskip

\textit{Case $u=0$.} Since $t_u-1= T(k-1,\ell)$, we can apply $P(k-1,\ell,t_u-1)$ to the $(t_u-1,2\delta)$-tree given by $\bar{x}'$, $\bar{y}'$, and $\bar{r}'$. If this is witnessed by (II), then the same maps composed with $\psi$ witness (II) in $P(k,\ell,t)$ for our initial tree given by $\bar{x}$, $\bar{y}$, and $\bar{r}$. So we may assume that $P(k-1,\ell,t_u-1)$ is witnessed by (I). Thus there are maps $\alpha'\colon[k-1]\to\bn{t_u-1}$ and $\beta'\colon [k-1]\to\bl{t_u-1}$ such that for all $i,j\in [k-1]$, if $i\leq j$ then
\begin{enumerate}[$(i)'$]
\item $f(x'_{\alpha'(j)},y'_{\beta'(i)})\geq r'_{\alpha'(j)}+\delta$ and
\item if $i<j$ then $f(x'_{\alpha'(i)},y'_{\beta'(j)})\leq r'_{\alpha'(i)}$.
\end{enumerate}
Define $\alpha\colon[k]\to\bn{t}$ and $\beta\colon [k]\to\bl{t}$ such that $\alpha(1)=\phi(\emp)$, $\beta(1)=\eta_*$, and for $1<i\leq k$, $\alpha(i)=\psi(\alpha'(i-1))$ and $\beta(i)=\psi(\beta'(i-1))$. 
We show that $\alpha$ and $\beta$ witness (I) in $P(k,\ell,t)$ for $\bar{x}$, $\bar{y}$, and $\bar{r}$. 

Fix $i,j\in[k]$ with $i\leq j$. To verify (I), we need to show:
\begin{enumerate}[$(i)$]
\item $f(x_{\alpha(j)},y_{\beta(i)})\geq r_{\alpha(j)}+\delta$ and
\item if $i<j$ then $f(x_{\alpha(i)},y_{\beta(j)})\leq r_{\alpha(i)}$.
\end{enumerate}

First suppose $i\geq 2$. Then 
\begin{multline*}
f(x_{\alpha(j)},y_{\beta(i)})
=f(x_{\psi(\alpha'(j-1))},y_{\psi(\beta'(i-1))})\\
=f(x'_{\alpha'(j-1)},y'_{\beta'(i-1)})
\geq r'_{\alpha'(j-1)}+\delta
=r_{\psi(\alpha'(j-1))}+\delta
=r_{\alpha(j)}+\delta,
\end{multline*}
where the inequality holds by $(i)'$ since $i-1\leq j-1$. 
This yields $(i)$ in this case. The verification of $(ii)$ follows similarly using $(ii)'$.

Suppose now $i=1$. Recall $u=0$. For $(i)$, first note that if $j=1$ then $\alpha(j)=\phi(\emp)$, while if $j>1$ then 
\[
\alpha(j)=\psi(\alpha'(j-1))=\phi^*(0\conc\alpha'(j-1))=\phi(0\conc\alpha'(j-1)),
\]
where the last equality uses the fact that $\alpha'(j-1)\in\bn{t_0-1}$ and hence $u\conc\alpha'(j-1)\in\bn{t_0}$. In either case, $\alpha(j)\in\phi(\bn{t_0})\seq P_0$, and thus $f(x_{\alpha(j)},y_{\beta(1)})
=
f(x_{\alpha(j)},y_{\eta_*})
\geq r_{\alpha(j)}+\delta$. This establishes $(i)$ (in the $i=1$ case).

Finally, for $(ii)$ (in the $i=1$ case), suppose $j>1$. Then
\[
\beta(j)=\psi(\beta'(j-1))=\phi^*(0\conc\beta'(j-1)).
\]
Since $\phi^*$ is a tree embedding extending $\phi$, we therefore have $\phi(\emp)\conc 0\iseq \beta(j)$. Thus, since $\bar{x}$, $\bar{y}$, and $\bar{r}$ form a $(t,2\delta)$-tree for $f$, we have
\[
f(x_{\alpha(1)},y_{\beta(j)})
=
f(x_{\phi(\semp)},y_{\beta(j)})
\leq r_{\phi(\semp)}
=
r_{\alpha(1)}.
\]

\textit{Case $u=1$.} The argument is similar to the previous case. We apply $P(k,\ell-1,t_u-1)$ to $\bar{x}'$, $\bar{y}'$, and $\bar{r}'$, and use the analogous argument to assume this is witnessed by (II). Thus there are maps $\alpha'\colon[\ell-1]\to\bn{t_u-1}$ and $\beta'\colon [\ell-1]\to\bl{t_u-1}$ such that for all $i,j\in [\ell-1]$, if $i\leq j$ then
\begin{enumerate}[$(i)'$]
\item $f(x'_{\alpha'(j)},y'_{\beta'(i)})<r'_{\alpha'(j)}+\delta$, and
\item if $i<j$ then $f(x'_{\alpha'(i)},y'_{\beta'(j)})\geq r'_{\alpha'(i)}+2\delta$.
\end{enumerate}
Define $\alpha\colon[\ell]\to\bn{t}$ and $\beta\colon [\ell]\to\bl{t}$ exactly as in the previous case, except with $k$ replaced by $\ell$.
We show that $\alpha$ and $\beta$ witness (II) in $P(k,\ell,t)$ for $\bar{x}$, $\bar{y}$, and $\bar{r}$. 

Fix $i,j\in[\ell]$ with $i\leq j$. To verify (II), we need to show:
\begin{enumerate}[$(i)$]
\item $f(x_{\alpha(j)},y_{\beta(i)})<r_{\alpha(j)}+\delta$ and
\item if $i<j$ then $f(x_{\alpha(i)},y_{\beta(j)})\geq r_{\alpha(i)}+2\delta$.
\end{enumerate}

If $i\geq 2$ then the verification of $(i)$ and $(ii)$ follows using $(i)'$ and $(ii)'$ as in the previous case. So we may assume $i=1$. Recall $u=1$. For $(i)$, as in the previous case, we have $\alpha(j)\in\phi(\bn{t_1})\seq P_1$, and thus $
f(x_{\alpha(j)},y_{\beta(1)})
=
f(x_{\alpha(j)},y_{\eta_*})
<
r_{\alpha(j)}+\delta$, as desired. For $(ii)$, suppose $j>1$. Then
\[
\beta(j)=\psi(\beta'(j-1))=\phi^*(1\conc\beta'(j-1)).
\]
Since $\phi^*$ is a tree embedding extending $\phi$, we therefore have $\phi(\emp)\conc 1\iseq \beta(j)$. Thus, since $\bar{x}$, $\bar{y}$, and $\bar{r}$ form a $(t,2\delta)$-tree for $f$, we have
\[
f(x_{\alpha(1)},y_{\beta(j)})
=
f(x_{\phi(\semp)},y_{\beta(j)})
\geq r_{\phi(\semp)}+2\delta
=
r_{\alpha(1)}+2\delta.\clqed
\]

Now, for $k,\ell\geq 1$, set $B(k,\ell)=2\binom{k+\ell-2}{k-1}-1$. Then for any $k,\ell\geq 1$, $B(k,1)=1=B(1,\ell)$. Moreover, if $k,\ell\geq 2$ then $B(k,\ell)=B(k-1,\ell)+B(k,\ell-1)+1$ 
by the standard recursive identity for binomial coefficients. By induction on $k+\ell$, it follows that $T(k,\ell)\leq B(k,\ell)$ for all $k,\ell\geq 1$.

We can now finish the proof of the theorem. Set $t=\binom{2k}{k}-1$. 
Suppose $f$ admits a $(t,2\delta)$-tree given by $(x_{\sigma}: \sigma \in \bn{t})$, $(y_{\eta}: \eta \in \bl{t})$, and $(r_\sigma:\sigma\in\bn{t})$. We show that $f$ admits a $(k,\delta)$-ladder. First, note that $t=B(k,k+1)$. So $t\geq T(k,k+1)$ by the above, and hence we can apply $P(k,k+1,t)$ to this $(t,2\delta)$-tree. 

Suppose first that (I) holds in $P(k,k+1,t)$, witnessed by $\alpha\colon[k]\to\bn{t}$ and $\beta\colon [k]\to \bl{t}$. For $i\in[k]$, set $a_i=x_{\alpha(k-i+1)}$, $b_i=y_{\beta(k-i+1)}$, and $r^*_i=r_{\alpha(k-i+1)}$. 
Given $i,j\in[k]$, if $i\leq j$ then by (I),
\[
f(a_i,b_j)
=
f(x_{\alpha(k-i+1)},y_{\beta(k-j+1)})
\geq r_{\alpha(k-i+1)}+\delta
=
r^*_i+\delta,
\]
while if $i>j$ then  by (I),
\[
f(a_i,b_j)
=
f(x_{\alpha(k-i+1)},y_{\beta(k-j+1)})
\leq r_{\alpha(k-i+1)}
=
r^*_i.
\]
Thus $(a_1,\ldots,a_k)$ and $(b_1,\ldots,b_k)$ yield a $(k,\delta)$-ladder for $f$ with values $(r^*_1,\ldots,r^*_k)$.

Finally suppose (II) holds in $P(k,k+1,t)$, witnessed by $\alpha\colon[k+1]\to\bn{t}$ and $\beta\colon [k+1]\to \bl{t}$. For $i\in [k]$, set $a_i=x_{\alpha(i)}$, $b_i=y_{\beta(i+1)}$, and $r^*_i=r_{\alpha(i)}+\delta$. Then a similar verification shows that $(a_1,\ldots,a_k)$ and $(b_1,\ldots,b_k)$ form a $(k,\delta)$-ladder for $f$ with values $(r^*_1,\ldots,r^*_k)$. 
\end{proof}

By the previous proof, we see that Theorem \ref{thm:hodgesfn1} also holds with the uniform variations of trees and ladders. Indeed, if the original tree in the proof has the uniform value $r$, then the resulting ladder has uniform value either $r$ or
$r+\delta$.  We record this observation in the following corollary.

\begin{corollary}\label{cor:hodgesfn1uni}
Given $k\geq 1$ and $\delta>0$, if $f\colon X\times Y\to [0,1]$ admits a uniform $(\binom{2k}{k}-1,2\delta)$-tree, then $f$ admits a uniform $(k,\delta)$-ladder. 
\end{corollary}

\begin{remark}\label{rem:2kkbound}
Using a standard approximation for central binomial coefficients, we see that the bound in Theorem \ref{thm:hodgesfn1} grows  on the order of $4^k/\sqrt{\pi k}$. By contrast, in the discrete case, the corresponding bound from Theorem \ref{thm:hodges}$(2)$ is on the order of $2^k$. This motivates the following open question.
\end{remark}

\begin{question}
In Theorem \ref{thm:hodgesfn1}, can $t$ be bounded on the order of $2^k$?\footnote{A ChatGPT query resulted in various tricks to improve the bound by a constant factor, but did not successfully resolve this question one way or the other.}
\end{question}

On the other hand, the $2\delta$ term in Theorem \ref{thm:hodgesfn1} cannot be improved: 

\begin{proposition}\label{prop:GPT1}
Fix $0<\delta\leq\frac{1}{4}$ and $t\geq 1$. Then for any $0<\alpha<2\delta$, there is a function $f\colon X\times Y\to [0,1]$ that admits a uniform $(t,\alpha)$-tree, but omits $(3,\delta)$-ladders. 
\end{proposition}
\begin{proof}
Set $X=\bn{t}$ and $Y=\bl{t}$. Define  $f\colon X\times Y\to [0,1]$ such that
\[
f(\sigma,\eta)=
\begin{cases}
0 & \text{if } \sigma\conc 0\iseq \eta,\\
\alpha & \text{if } \sigma\conc 1\iseq \eta,\\
{\textstyle\frac{\alpha}{2}} & \text{otherwise.}
\end{cases}
\]
Then $f$ admits a uniform $(t,\alpha)$-tree with value $0$. Toward a contradiction, suppose $f$ admits a $(3,\delta)$-ladder with sequences $(\sigma_1,\sigma_2,\sigma_3)$ and $(\eta_1,\eta_2,\eta_3)$, and values $(r_1,r_2,r_3)$. \medskip

\noindent\textit{Claim 1.} If $i,j,j'\in[3]$ and $j'<i\leq j$, then $\sigma_i\conc 1\iseq \eta_j$ and $\sigma_i\conc 0\iseq \eta_{j'}$. 

\noindent\textit{Proof.} The definition of a $(3,\delta)$-ladder yields $f(\sigma_i,\eta_j)\geq r_i+\delta$ and $f(\sigma_i,\eta_{j'})\leq r_i$, hence $f(\sigma_i,\eta_j)\geq f(\sigma_i,\eta_{j'})+\delta$. Since $f$ is $\{0,\alpha,\frac{\alpha}{2}\}$-valued, and $\frac{\alpha}{2}<\delta$, we must have $f(\sigma_i,\eta_j)=\alpha$ and $f(\sigma_i,\eta_{j'})=0$. The claim now follows by definition of $f$.\clqed\medskip

Now, given $\sigma\in \bne{t}$, let $C(\sigma)=\{\eta\in \bl{t}:\sigma\iseq \eta\}$. Note that if $\sigma,\tau\in \bne{t}$ and $C(\sigma)\cap C(\tau)\neq\emptyset$, then either $\sigma\iseq \tau$ in which case $C(\tau)\seq C(\sigma)$, or $\tau\iseq\sigma$ in which case $C(\sigma)\seq C(\tau)$. On the other hand, setting $C=C(\sigma_2\conc 1)$ and $C'=C(\sigma_3\conc 0)$, Claim 1 implies  $\eta_2\in C\cap C'$, $\eta_3\in C\backslash C'$, and $\eta_1\in C'\backslash C$, which is a contradiction. 
\end{proof}

Finally, we prove the two corollaries stated in Subsection \ref{sec:mainresults}. The first is an implication from trees to uniform ladders, whose proof passes first through the implication from trees to uniform trees in Proposition \ref{prop:unitreeimp}.

\begin{proof}[\textnormal{\textbf{Proof of Corollary \ref{cor:hodgesfn1}}}]
Fix $k\geq 1$ and $\delta,\e>0$. 
Set $t=\lceil \e\inv\rceil(\binom{2k}{k}-2)+1$, and note that $t\leq \lceil\e\inv \rceil 4^k$ since  $\binom{2k}{k}\leq 4^k$. Suppose $f\colon X\times Y\to [0,1]$ admits a $(t,2\delta+\e)$-tree. Then $f$ admits a uniform $(\binom{2k}{k}-1,2\delta)$-tree by Proposition \ref{prop:unitreeimp}, and hence a uniform $(k,\delta)$-ladder  by Corollary \ref{cor:hodgesfn1uni}.
\end{proof}

Recall that Theorem \ref{thm:AB} (from \cite{AndBen}) provides an implication from a $(t,2\delta+\e)$-tree to an agnostic $(k,\delta)$-ladder, with $t\leq O((1/\e)^k)$, and hence Corollary \ref{cor:hodgesfn1} constitutes a stronger conclusion with an improved bound. It is interesting to note that to prove this result, one might be inclined to instead first use Theorem \ref{thm:hodgesfn1} to extract a ladder from a tree, and then uniformize the ladder. This strategy is possible (see Proposition \ref{prop:lad-unilad}), but it would result in a  bound on the order of $4^{k/\e}$, which is worse than Theorem \ref{thm:AB}.

Finally, we prove Corollary \ref{cor:dualbound}, which concerns transferring trees in  $f\colon X\times Y\to [0,1]$ to trees in $f^{opp}$. Recall that the discrete version of this argument (sketched after Theorem \ref{thm:hodges}) relied on the fact that $k$-ladders for some $E\seq X\times Y$ can be reindexed to be $k$-ladders for $E^{opp}$. In the case of functions, the argument is  a little more subtle since a $(k,\delta)$-ladder does not satisfy the same symmetry between $f$ and $f^{opp}$ due to the fact that the values of the ladder are attached to elements from $X$. Therefore, we must pass through uniform ladders, which introduces the ``slack" parameter $\e$. For later purposes, we record this symmetry observation separately, and then move on to the proof of Corollary \ref{cor:dualbound}.

\begin{remark}\label{rem:uniformlad-sym}
Fix $k\geq 1$, $\delta>0$, and $f\colon X\times Y\to [0,1]$. Suppose $x_1,\ldots,x_k\in X$ and $y_1,\ldots,y_k\in Y$ form a uniform $(k,\delta)$-ladder for $f$ with value $r$. Then $(y_k,\ldots,y_1)$ and $(x_k,\ldots,x_1)$ form a uniform $(k,\delta)$-ladder for $f^{opp}$ with value $r$.
\end{remark}

\begin{proof}[\textnormal{\textbf{Proof of Corollary \ref{cor:dualbound}}}]
Fix $t\geq 1$ and $\delta,\e>0$, and let $t^*\leq \lceil\e\inv\rceil2^{2^{t+1}}$ be as in Corollary \ref{cor:hodgesfn1} with $k=2^t$. Suppose $f\colon X\times Y\to [0,1]$ admits a $(t^*,2\delta+\e)$-tree. By Corollary \ref{cor:hodgesfn1}, $f$ admits a uniform $(2^t,\delta)$-ladder, which then gives a $(2^t,\delta)$-ladder in $f^{opp}$ by Remark \ref{rem:uniformlad-sym}. So $f^{opp}$ admits a $(t,\delta)$-tree by Theorem \ref{thm:optree}. 
\end{proof}

Using the sharper approximation of central binomial coefficients in Remark \ref{rem:2kkbound}, one can slightly improve the bound in Corollary \ref{cor:dualbound} to $t^*\sim \frac{1}{\sqrt{\pi}}\lceil \e\inv\rceil 2^{2^{t+1}-t/2}$ (with $\e$ fixed and $t$ tending to infinity).

\section{Improved extraction of tight ladders from trees}\label{sec:tight}

The goal of this section is to prove Theorem \ref{thm:tight}. The overall proof structure is similar to that of Theorem \ref{thm:hodgesfn1} in that we will extract a ladder-like configuration from a tree, using the tree coloring lemmas proved in Subsection \ref{sec:ramsey}. Let us first define said ladder-like configuration.

\begin{definition}\label{def:proto}
Fix $k\geq 1$, $\delta>0$, and collections $\calI$ and $\calJ$ of subsets of $[0,1]$. An \emph{$(\calI,\calJ)$-covered $(k,\delta)$-proto-ladder} for a function $f\colon X\times Y\to [0,1]$ consists of  sequences  $(x_1,\ldots,x_k)$ from $X$, $(y_1,\ldots,y_{k+1})$ from $Y$ and $(r_1,\ldots,r_k)$ from $[0,1]$ such that, for some $I_1,\ldots,I_k\in \calI$ and $J_1,\ldots,J_k\in\calJ$, the following properties hold for all $i\in[k]$.
\begin{enumerate}
\item[(1)$_i$] For all $j\in[k]$, if  $i\leq j$ then $f(x_j,y_i)\in I_i$.
\item[(2)$_i$] For all $j\in[k+1]$, if $i<j$ then $f(x_i,y_j)\in J_i$.
\item[(3)$_i$] One of the following properties holds.
\begin{enumerate}
\item[(I)$_i$] $f(x_i,y_i)>r_i$ and, for all $j\in[k+1]$, if $i<j$ then $f(x_i,y_j)\leq r_i-\delta$.
\item[(II)$_i$] $f(x_i,y_i)\leq r_i$ and, for all $j\in[k+1]$, if $i<j$ then $f(x_i,y_j)\geq r_i+\delta$.
\end{enumerate}
\end{enumerate}
When such sequences exist, we say $f$ \emph{admits an $(\calI,\calJ)$-covered $(k,\delta)$-proto-ladder}. 
If $\calI=\{I\}$ and $\calJ=\{J\}$ for some $I,J\seq [0,1]$, then we write $I$ and $J$ rather than $\calI$ and $\calJ$.
\end{definition}

The previous definition is an elaboration on the kind of configuration constructed in the proof of Theorem \ref{thm:hodgesfn1}. Note that conditions (I) and (II) in that proof loosely match conditions (I)$_i$ and (II)$_i$ in property (3)$_i$ of Definition \ref{def:proto}. However, in the proof of Theorem \ref{thm:hodgesfn1}, we used Lemma \ref{lem:ramsey} to preserve either (I) or (II)  globally, which then implied the existence of a ladder. In the present situation, we will instead use Lemmas \ref{lem:ramsey} and \ref{lem:ramseyleaves} in parallel to accomplish properties (1)$_i$ and (2)$_i$, which lay the groundwork for tightness. We will then apply pigeonhole to pass to a subsequence with uniform choices of $I$ and $J$ in (1)$_i$ and (2)$_i$. By choosing $I$ and $J$ to be sufficiently small intervals, we will force either (I)$_i$ or (II)$_i$ to hold globally for all $i$, which altogether will yield a uniform tight ladder. The next two lemmas extract the main technical steps of this rough sketch. First, we show that a proto-ladder uniformly covered by small intervals contains a uniform tight ladder. 

\begin{lemma}\label{lem:protoladder}
Fix $k\geq 1$ and $\delta,\e>0$. Suppose $f\colon X\times Y\to [0,1]$ admits an $(I,J)$-covered $(k,\delta+\e)$-proto-ladder, where $I$ and $J$ are intervals of length at most $\e$. Then $f$ admits a uniform $\e$-tight $(k,\delta)$-ladder.
\end{lemma}
\begin{proof}
We first deal with the $k=1$ case. Note that in order to construct a uniform $\epsilon$-tight $(1,\delta)$-ladder, one only needs $x\in X$ and $y\in Y$ such that $f(x,y)\geq\delta$, which easily follows from the existence of an $(I,J)$-covered $(1,\delta+\e)$-proto-ladder. So we may assume $k\geq 2$.

Fix $x_1,\ldots,x_k\in X$, $y_1,\ldots,y_{k+1}\in Y$, and $r_1,\ldots,r_k\in [0,1]$ comprising an $(I,J)$-covered $(k,\delta+\e)$-proto-ladder for $f$. Throughout the proof, we will refer to  properties (1)$_i$, (2)$_i$, and (3)$_i$ from  Definition \ref{def:proto}  in the context of  these sequences. In particular, note that $I_i=I$ in property (1)$_i$, and $J_i=J$ in property (2)$_i$.\medskip

\noindent\textit{Claim 1.} One of the following two cases holds:
\begin{enumerate}[(I)]
\item For all $i\in[k]$, property \textnormal{(I)}$_i$ holds in (3)$_i$.
\item For all $i\in[k]$, property \textnormal{(II)}$_i$ holds in (3)$_i$.
\end{enumerate}

\noindent\textit{Proof.} For $i\in[k]$, set $u_i=f(x_i,y_i)$ and $v_i=f(x_i,y_{k+1})$. Toward a contradiction, suppose the claim fails. Fix $i,j\in[k]$ such that (I)$_i$ holds and (II)$_j$ holds. By (I)$_i$, we have $u_i>r_i$ and $v_i\leq r_i-(\delta+\e)$, hence $u_i-v_i>\delta+\e$. Similarly, by (II)$_j$, we have $u_j\leq r_j$ and $v_j\geq r_j+\delta+\e$, hence $v_j-u_j\geq\delta+\e$. Therefore
\begin{equation}\label{eq:2rhofii}
2\delta+2\e<u_i-v_i+v_j-u_j=u_i-u_j+v_j-v_i.
\end{equation}
On the other hand, $u_i,u_j\in I$ by  (1)$_i$ and (1)$_j$, while   $v_i,v_j\in J$ by (2)$_i$ and (2)$_j$, which implies  
\begin{equation}\label{eq:fiitau}
\max\{|u_i-u_j|,|v_i-v_j|\}\leq\e.
\end{equation}
Together, (\ref{eq:2rhofii}) and (\ref{eq:fiitau}) imply $2\delta+2\e<2\e$, which is a contradiction.\clqed\medskip

We now define sets $U,V\seq[0,1]$ and sequences $(a_1,\ldots,a_k)$ from $X$ and  $(b_1,\ldots,b_k)$ from $Y$ according to the cases in Claim 1.\medskip

\noindent Case (I): Set $U=I$ and $V=J$. For $i\in [k]$, set $a_i=x_{k-i+1}$ and $b_i=y_{k-i+1}$.

\noindent Case (II): Set $U=J$ and $V=I$. For $i\in [k]$, set $a_i=x_i$ and $b_i=y_{i+1}$.\medskip

We will show that $(a_1,\ldots,a_k)$ and $(b_1,\ldots,b_k)$ form an $\e$-tight $(k,\delta)$-ladder for $f$. \medskip

\noindent\textit{Claim 2.} 
\begin{enumerate}[$(a)$]
\item Given $i,j\in[k]$, if $i\leq j$ then $f(a_i,b_j)\in U$.
\item Given $i,j\in[k]$, if $i>j$ then $f(a_i,b_j)\in V$. 
\end{enumerate}
\noindent\textit{Proof.} 
First assume case (I). For $(a)$, if $i\leq j$ then $f(a_i,b_j)=f(x_{k-i+1},y_{k-j+1})\in I$ by  (1)$_{k-j+1}$. For $(b)$, if $i>j$ then $f(a_i,b_j)=f(x_{k-i+1},y_{k-j+1})\in J$ by  (2)$_{k-i+1}$.

Now assume case (II). For $(a)$, if $i\leq j$ then $f(a_i,b_j)=f(x_i,y_{j+1})\in J$ by  (2)$_i$. For $(b)$, if $i>j$ then $f(a_i,b_j)=f(x_i,y_{j+1})\in I$ by  (1)$_{j+1}$.\clqed\medskip

Now define
\[
r=\max\{f(a_i,b_j):i,j\in[k],~i>j\} \quad \text{and} \quad s=\min\{f(a_i,b_j):i,j\in[k],~i\leq j\}.\footnote{Note that $r$ exists since $k\geq 2$.}
\]

\noindent\textit{Claim 3.} $r+\delta\leq s$.

\noindent\textit{Proof.} Fix $i,j,i',j'\in[k]$ with $i\leq j$ and $i'>j'$. We need to show
\begin{equation}\label{eq:srsep0}
f(a_i,b_j)-f(a_{i'},b_{j'})\geq\delta.
\end{equation}

First assume case (I). Set $i^*=k-i'+1$ and $j^*=k-j'+1$. Then $f(a_i,b_j)\in I$ by Claim 2, and $f(x_{i^*},y_{i^*})\in I$ by (1)$_{i^*}$. Thus
\begin{equation}\label{eq:srsep1}
f(a_i,b_j)\geq f(x_{i^*},y_{i^*})-\e>r_{i^*}-\e,
\end{equation}
where the final inequality uses (I)$_{i^*}$. 
On the other hand, since $i^*<j^*$, by (I)$_{i^*}$ we also have
\begin{equation}\label{eq:srsep2}
f(a_{i'},b_{j'})=f(x_{i^*},y_{j^*})\leq r_{i^*}-(\delta+\e).
\end{equation}
Together, (\ref{eq:srsep1}) and (\ref{eq:srsep2}) imply (\ref{eq:srsep0}).

Now assume case (II). Then $f(a_{i'},b_{j'})\in I$ by Claim 2, and $f(x_i,y_i)\in I$ by  (1)$_i$. Thus 
\begin{equation}\label{eq:srsep3}
f(a_{i'},b_{j'})\leq f(x_i,y_i)+\e\leq r_i+\e,
\end{equation}
where the final inequality uses (II)$_i$. On the other hand, by (II)$_i$ we have
\begin{equation}\label{eq:srsep4}
f(a_i,b_j)=f(x_i,y_{j+1})\geq r_i+\delta+\e.
\end{equation}
Together, (\ref{eq:srsep3}) and (\ref{eq:srsep4}) imply (\ref{eq:srsep0}).\clqed\medskip

Finally, by Claim 3, $(a_1,\ldots,a_k)$ and $(b_1,\ldots,b_k)$ form a uniform $(k,\delta)$-ladder for $f$ with value $r$ and, moreover, this ladder is $\e$-tight by Claim 2.
\end{proof}

Next, we  extract proto-ladders  from trees. 


\begin{lemma}\label{lem:tighttree}
Fix finite covers $\calI$ and $\calJ$ of $[0,1]$. Set $p=|\calI|$ and $q=|\calJ|$. For $k\geq 1$, set 
\[
t_k=1+pq+(pq)^2+\ldots+(pq)^{k-1}.
\]
Then for any $k\geq 1$ and $\delta>0$,  if $f\colon X\times Y\to [0,1]$ admits a $(t_k,2\delta)$-tree, then $f$ admits an $(\calI,\calJ)$-covered $(k,\delta)$-proto-ladder.
\end{lemma}
\begin{proof}
To ease notation, given arbitrary functions $\phi\colon U\to V$ and $\psi \colon V\to W$, we use concatenation $\psi\phi$ for the composition $\psi\circ \phi\colon U\to W$. 

Fix $\delta>0$ and $f\colon X\times Y\to [0,1]$. Given $k\geq 1$, let $P(k)$ be the following statement: \medskip

\noindent\textit{Statement of $P(k)$.} If $(x_\sigma: \sigma\in\bn{t_k})$ and $(y_\eta:\eta\in\bl{t_k})$ form a $(t_k,2\delta)$-tree for $f$ with values $(r_\sigma:\sigma\in\bn{t_k})$,   then there are maps $\alpha\colon[k]\to \bn{t_k}$ and $ \beta\colon [k+1]\to \bl{t_k}$ such that $(x_{\alpha(1)},\ldots,x_{\alpha(k)})$, $(y_{\beta(1)},\ldots,y_{\beta(k+1)})$, and $(r_{\alpha(1)}+\delta,\ldots,r_{\alpha(k)}+\delta)$ form an $(\calI,\calJ)$-covered $(k,\delta)$-proto-ladder for $f$. \medskip

We prove $P(k)$ holds for all $k\geq 1$ by induction. \medskip

\noindent\textit{Base Case:} Note that $t_1=1$. Suppose $f$ admits a $(1,2\delta)$-tree given by $x_{\semp}\in X$, $y_0,y_1\in Y$, and $r_{\semp}\in[0,1]$. Since $f(x_{\semp},y_0)\leq r_{\semp}$ and $f(x_{\semp},y_1)\geq r_{\semp}+2\delta$, one easily verifies that $x_{\semp}$, $(y_0,y_1)$, and $r_{\semp}+\delta$ form an $(\calI,\calJ)$-covered $(1,\delta)$-proto-ladder for $f$, witnessed by any $I_1\in\calI$ containing $f(x_{\semp},y_0)$ and any $J_1\in\calJ$ containing $f(x_{\semp},y_1)$. So, in particular, we may set $\alpha\colon[1]\to \bn{1}$ and $\beta\colon[2]\to\bl{1}$ such that $\alpha(1)=\emp$, $\beta(1)=0$, and $\beta(2)=1$.\medskip  

\noindent\textit{Induction Step:} Fix $k>1$ and suppose $P(k-1)$ holds.  Set $\ell=qt_{k-1}+1$ and note that $t_k=p\ell-p+1$ by the definition of $t_k$. Suppose 
\[
\bar{x}^*=(x_\sigma: \sigma\in\bn{t_k}),\quad \bar{y}^*=(y_\eta:\eta\in\bl{t_k}),\quad\text{and}\quad \bar{r}^*=(r_\sigma:\sigma\in\bn{t_k})
\]
 form a $(t_k,2\delta)$-tree for $f$. Fix any $\eta_*\in \bl{t_k}$. For $I\in\calI$, define
\[
C_I=\{\sigma\in\bn{t_k}:f(x_\sigma,y_{\eta_*})\in I\}.
\]
Since $t_k=p\ell-p+1$, we may apply Lemma \ref{lem:ramsey} to obtain $I_1\in\calI$ and a tree embedding $\phi\colon \bn{\ell}\to \bn{t_k}$ whose image is contained in $C_{I_1}$. By Proposition \ref{prop:leavesemb}$(d)$, we can extend $\phi$ to a proper tree embedding $\phi^*\colon \bne{\ell}\to \bne{t_k}$. 

Set
\[
u=\begin{cases}
0 & \text{if $f(x_{\phi(\semp)},y_{\eta_*})>r_{\phi(\semp)}+\delta$}\\
1 & \text{if $f(x_{\phi(\semp)},y_{\eta_*})\leq r_{\phi(\semp)}+\delta$.}
\end{cases}
\]
Define $\psi\colon \bn{\ell}\to\bne{t_k}$ such that $\psi(\sigma)=\phi^*(u\conc\sigma)$. Then  $\psi$ is a proper tree embedding by Proposition \ref{prop:leavesemb}$(b)$ and the fact that $\sigma\mapsto u\conc\sigma$ is a proper tree embedding from $\bn{\ell}$ to $\bne{\ell}$. 
For $J\in\calJ$, define 
\[
\tilde{C}_J=\{\eta\in \bl{\ell-1}:f(x_{\phi(\semp)},y_{\psi(\eta)})\in J\}.
\]
Since $\ell-1=qt_{k-1}$, we may apply Lemma \ref{lem:ramseyleaves} to obtain $J_1\in\calJ$ and a proper tree embedding $\theta\colon \bne{t_{k-1}}\to\bn{\ell}$ such that $\theta(\bl{t_{k-1}})\seq \tilde{C}_{J_1}$. 

Now set $\chi=\psi\theta\colon \bne{t_{k-1}}\to\bne{t_k}$ (recall our use of concatenation for composition of functions). Then $\chi$ is a proper tree embedding  by Proposition \ref{prop:leavesemb}$(b)$. Therefore, by Proposition \ref{prop:embedtree},  the sequences
\[
(x_{\chi(\sigma)}:\sigma\in\bn{t_{k-1}}),\quad (y_{\chi(\eta)}:\eta\in \bl{t_{k-1}}),\quad\text{and}\quad (r_{\chi(\sigma)}:\sigma\in\bn{t_{k-1}})
\]
 form a $(t_{k-1},2\delta)$-tree for $f$. By $P(k-1)$, there are  $\alpha'\colon [k-1]\to \bn{t_{k-1}}$ and $\beta'\colon [k]\to \bl{t_{k-1}}$ such that 
 \[
\bar{x}'=(x_{\chi\alpha'(1)},\ldots,x_{\chi\alpha'(k-1)}),\quad \bar{y}'=(y_{\chi\beta'(1)},\ldots, y_{\chi\beta'(k)}),\quad\text{and}\quad \bar{r}'=(r_{\chi\alpha'(1)}+\delta,\ldots,r_{\chi\alpha'(k-1)}+\delta)
\]
 form an $(\calI,\calJ)$-covered $(k-1,\delta)$-proto-ladder for $f$, say witnessed by $I'_1,\ldots,I'_{k-1}\in\calI$ and $J'_1,\ldots,J'_{k-1}\in\calJ$.
Define $\alpha\colon [k]\to\bn{t_k}$ and $\beta\colon [k+1]\to \bl{t_k}$ such that 
\[
\alpha(i)=\begin{cases}
\phi(\emp) & \text{if $i=1$,}\\
\chi\alpha'(i-1) & \text{if $1<i\leq k$}
\end{cases}
\quad\text{and}\quad
\beta(i) =\begin{cases}
\eta_* & \text{if $i=1$,}\\
\chi\beta'(i-1) & \text{if $1<i\leq k+1$.}
\end{cases}
\]
 For $1<i\leq k$, set $I_i=I'_{i-1}$ and $J_i=J'_{i-1}$.  To finish the induction step, we show that 
 \[
 \bar{x}=(x_{\alpha(1)},\ldots,x_{\alpha(k)}),\quad \bar{y}=(y_{\beta(1)},\ldots,y_{\beta(k+1)}),\quad\text{and}\quad\bar{r}=(r_{\alpha(1)}+\delta,\ldots,r_{\alpha(k)}+\delta)
 \]
 form an $(\calI,\calJ)$-covered $(k,\delta)$-proto-ladder for $f$, witnessed by $I_1,\ldots,I_k$ and $J_1,\ldots,J_k$.
 
 For clarity, we write (1)$_i$, (2)$_i$, and (3)$_i$ for the properties in Definition \ref{def:proto} that we need to verify for $\bar{x}$, $\bar{y}$, and $\bar{r}$; and we write (1)$_i'$, (2)$_i'$, and (3)$'_i$ for the properties in Definition \ref{def:proto} that we know hold of $\bar{x}'$, $\bar{y}'$, and $\bar{r}'$. 
 
Fix $i\in[k]$. We need to verify (1)$_i$, (2)$_i$ and (3)$_i$. If $i\geq 2$ then these follow directly from (1)$'_{i-1}$, (2)$'_{i-1}$, and (3)$'_{i-1}$ by definition of $\alpha(i)$, $\beta(i)$, $I_i$, and $J_i$. So we may assume $i=1$. 

We first verify (1)$_1$. Fix $j\in [k]$. It suffices to show $\alpha(j)\in C_{I_1}$ since this will imply
\[
f(x_{\alpha(j)},y_{\beta(1)})=f(x_{\alpha(j)},y_{\eta_*})\in I_1,
\]
as desired. Now recall that $\Im(\phi)\seq C_{I_1}$, and hence it suffices to show $\alpha(j)\in \Im(\phi)$. If $j=1$ this  holds by definition. So suppose $j>1$. Then 
\begin{equation}\label{eq:alphaj}
\alpha(j)=\chi\alpha'(j-1)=\psi\theta\alpha'(j-1)=\phi^*(u\conc \theta\alpha'(j-1)).
\end{equation}
Recall that $\alpha'(j-1)\in \bn{t_{k-1}}$ and thus, since $\theta$ is proper, we have  $\theta\alpha'(j-1)\in\bn{\ell-1}$ by Proposition \ref{prop:leavesemb}$(c)$. So $u\conc\theta\alpha'(j-1)\in \bn{\ell}$, which then yields $\alpha(j)\in \Im(\phi)$ by (\ref{eq:alphaj}) and the fact that $\phi^*$ extends $\phi$.

Next we verify (2)$_1$. Fix $j\in[k+1]$ with $j>1$.  Then $\beta'(j-1)\in \bl{t_{k-1}}$ and $\chi(\bl{t_{k-1}})=\psi\theta(\bl{t_{k-1}})\seq \psi(\tilde{C}_{J_1})$, hence $\beta(j)=\chi\beta'(j-1)=\psi(\eta)$ for some $\eta\in \tilde{C}_{J_1}$. Therefore
\[
f(x_{\alpha(1)},y_{\beta(j)})=f(x_{\phi(\semp)},y_{\psi(\eta)})\in J_1.
\]

Finally, we verify (3)$_1$.  Suppose first $u=0$. We show that (I)$_1$ holds. First, we have
\[
f(x_{\alpha(1)},y_{\beta(1)})=f(x_{\phi(\semp)},y_{\eta_*})>r_{\phi(\semp)}+\delta.
\]
Now fix $j\in[k+1]$ with $j>1$. Then
\[
\beta(j)=\chi\beta'(j-1)=\psi\theta\beta'(j-1)=\phi^*(0\conc\theta\beta'(j-1)).
\]
Since $\phi^*$ is a tree embedding extending $\phi$, we therefore have $\phi(\emp)\conc 0\iseq \beta(j)$. Thus, since our initial sequences $\bar{x}^*$, $\bar{y}^*$, and $\bar{r}^*$ form a $(t_k,2\delta)$-tree for $f$, we have
\[
f(x_{\alpha(1)},y_{\beta(j)})=f(x_{\phi(\semp)},y_{\beta(j)})\leq r_{\phi(\semp)}=(r_{\alpha(1)}+\delta)-\delta,
\]
as desired.

 Suppose now $u=1$. We show that (II)$_1$ holds. First, we have 
 \[
 f(x_{\alpha(1)},y_{\beta(1)})=f(x_{\phi(\semp)},y_{\eta_*})\leq r_{\phi(\semp)}+\delta.
 \]
 Now fix $j\in[k+1]$ with $j>1$. As in the previous case, $\beta(j)=\phi^*(1\conc\theta\beta'(j-1))$, hence $\phi(\semp)\conc 1\iseq\beta(j)$, which implies 
\[
f(x_{\alpha(1)},y_{\beta(j)})=f(x_{\phi(\semp)},y_{\beta(j)})\geq r_{\phi(\semp)}+2\delta=(r_{\alpha(1)}+\delta)+\delta,
\]
as desired.
\end{proof}

Finally, we combine Lemmas \ref{lem:protoladder} and \ref{lem:tighttree} to prove Theorem \ref{thm:tight}.

\begin{delayedproof}{thm:tight}
Fix $k\geq 1$, $\delta,\e>0$, and $f\colon X\times Y\to [0,1]$. To ease notation, we find $t\leq (1/\e)^{O(k/\e^2)}$ such that if $f$ admits a $(t,2\delta+2\e)$-tree, then $f$ admits a uniform $\e$-tight $(k,\delta)$-ladder.\footnote{Note  this is  technically a stronger statement than  Theorem \ref{thm:tight} (after replacing $\e$ with $\e/2$), but the difference only affects  the absolute constant in the bound on $t$.} 

Set $p=\lceil\e\inv\rceil$ and define the cover $\calI$ of $[0,1]$ consisting of intervals $[(i-1)p\inv,ip\inv]$ for $i\in [p]$. Note that $|\calI|=p$ and each interval in $\calI$ has length $p\inv\leq\e$. Set $k^*=p^2(k-1)+1$. Define
\[
t=1+p^2+p^4+\ldots+p^{2k^*-2}=\frac{p^{2k^*}-1}{p^2-1}\leq (1/\e)^{O(k/\e^2)}.
\]

Now suppose $f$ admits a $(t,2\delta+2\e)$-tree. By choice of $t$, we may apply Lemma \ref{lem:tighttree} to obtain an $(\calI,\calI)$-covered $(k^*,\delta+\e)$-proto-ladder for $f$ consisting of sequences
\[
(x_1,\ldots,x_{k^*}),\quad (y_1,\ldots,y_{k^*+1}),\quad (r_1,\ldots,r_{k^*}),\quad (I_1,\ldots,I_{k^*}),\quad\text{and}\quad(J_1,\ldots,J_{k^*}).
\]
 Since $k^*=p^2(k-1)+1$ and $|\calI\times\calI|=p^2$, we can apply the pigeonhole principle to extract intervals $I,J\in\calI$ and a subsequence $1\leq \ell_1<\ldots<\ell_k\leq k^*$ such that for all $i,j\in[k]$, $(I_{\ell_i},J_{\ell_i})=(I_{\ell_j},J_{\ell_j})= (I,J)$. One can now check that the subsequences
\[
(x_{\ell_1},\ldots,x_{\ell_k}),\quad (y_{\ell_1},\ldots,y_{\ell_k},y_{k^*+1}),\quad\text{and}\quad (r_{\ell_1},\ldots,r_{\ell_k})
\]
form an $(I,J)$-covered $(k,\delta+\e)$-proto-ladder for $f$. Since $I$ and $J$ are intervals of length at most $\e$, we conclude from Lemma \ref{lem:protoladder}  that $f$ admits a uniform $\e$-tight $(k,\delta)$-ladder. 
\end{delayedproof}

\section{Agnostic trees}\label{sec:agnostic}

In this section, we explore an additional variation of $(t,\delta)$-trees, which is more directly analogous to the notion of an agnostic ladder. 

\begin{definition}\label{def:eptree}
Suppose $f\colon X\times Y\rightarrow [0,1]$ is a function. Given $t\geq 1$ and $\delta>0$, an \emph{agnostic $(t,\delta)$-tree} for $f$ consists of sequences $(x_{\sigma}: \sigma\in \bn{t})$ from $X$ and $(y_{\tau}: \tau\in \bl{t})$ from $Y$ 
such that, for all $\sigma\in \bn{t}$ and $\eta, \eta'\in \bl{t}$, if  $\sigma\conc 0\iseq \eta$ and  $\sigma\conc 1\iseq \eta'$ then 
\[
|f(x_{\sigma},y_{\eta})-f(x_{\sigma},y_{\eta'})|\geq\delta.
\]
\end{definition}

When applied to the indicator function of a binary relation, Definition \ref{def:eptree} is the same as the original notion of a $t$-tree (up to relabeling the vertices), provided $\delta\leq 1$. However, this is not the case for $[0,1]$-valued functions. As indicated after Proposition \ref{prop:leavesemb}, this is the primary reason that we needed to modify Hodges' notion of a tree map into our notion of a tree embedding. 

That said, one can still establish an approximate equivalence between trees and agnostic trees. In particular, it is easy to see that a uniform $(t,\delta)$-tree for a function $f$ is an agnostic $(t,\delta)$-tree. The following gives a suitable converse. 

\begin{theorem}\label{thm:OPsharp}
Fix  $\delta,\e>0$ with $\e<1$. Given $t\geq 1$, define 
\[
T(t,\e)=\frac{\lceil 2\e\inv \rceil ^t-1}{\lceil 2\e\inv \rceil -1}.
\]
Then for any $t\geq 1$, if $f\colon X\times Y\to [0,1]$ admits an agnostic $(T(t,\e),\delta+\e)$-tree, then $f$ admits a  $(t,\delta)$-tree.
\end{theorem}
\begin{proof}
Fix $\delta,\e>0$ with $\e<1$, and fix a function $f\colon X\times Y\to [0,1]$. Given $t\geq 1$, let $P(t)$ be the following statement: If $f$ admits an agnostic $(T(t,\e),\delta+\e)$-tree, then $f$ admits a $(t,\delta)$-tree whose leaves are contained in the leaves of the initial agnostic tree. We show $P(t)$ holds for all $t\geq 1$ by induction. \medskip

\textit{Base Case}: Note that $T(1,\e)=1$.  It is easy to see that an agnostic $(1,\delta)$-tree for $f$ can be relabeled to be a $(1,\delta)$-tree. Thus the same is true of an agnostic $(1,\delta+\e)$-tree, which establishes $P(1)$.\medskip

\textit{Induction Step}: Fix $t>1$ and assume $P(t-1)$ holds. Set $T=T(t,\e)$ and suppose $f$ admits an agnostic $(T,\delta+\e)$-tree  with nodes $(x_{\sigma}: \sigma\in \bn{T})$ and leaves $(y_{\eta}: \eta\in \bl{T})$.  Set $K=\lceil 2\e\inv\rceil$ and $\Sigma=\{\frac{i}{K}:i\in[K]\}$. Note that $T=(K^t-1)/(K-1)$. \medskip

\noindent\textit{Claim 1.} There are $\alpha_0,\alpha_1\in \Sigma$ and, for each $u\in\{0,1\}$, a $(t-1,\delta)$-tree for $f$ with nodes $(c_{\sigma}^u: \sigma\in \bn{t-1})$, leaves $(d_{\eta}^u: \eta\in \bl{t-1})$, and values $(r^u_{\sigma}: \sigma\in \bn{t-1})$ satisfying the following properties:
\begin{enumerate}[$(i)$]
\item For all $u\in \{0,1\}$ and $\eta\in \bl{t-1}$, $d^u_\eta=y_{u\sconc\eta'}$ for some $\eta'\in \bl{T-1}$.
\item For all $u\in\{0,1\}$ and $\eta\in\bl{t-1}$, $\alpha_u-\frac{1}{K}\leq f(x_{\semp},d^u_\eta)\leq \alpha_u$.
\item $|\alpha_0-\alpha_1|\geq \delta+\e-\frac{1}{K}$.
\end{enumerate}

\noindent\textit{Proof.} For each $u\in\{0,1\}$, $\sigma\in \bn{T-1}$, and $\eta\in\bl{T-1}$, set $a^u_{\sigma}=x_{u\sconc\sigma}$ and $b^u_{\eta}=y_{u\sconc \eta}$. For each $u\in\{0,1\}$, define the cover 
\[
\bl{T-1}=\bigcup_{\alpha\in\Sigma}C^u_\alpha
\]
where, for each $\alpha\in\Sigma$,
$$
C^u_{\alpha}:=\left\{\eta\in \bl{T-1}: \alpha-{\textstyle\frac{1}{K}}\leq f(x_{\semp},b_{\eta}^u)\leq  \alpha\right\}.
$$
Set $T'=T(t-1,\e)$. One can check that $KT'=T-1$. So  Lemma \ref{lem:ramseyleaves} applied to the above covers  yields, for each $u\in \{0,1\}$, some $\alpha_u\in \Sigma$ and a proper tree embedding $\phi_u:\bne{T'}\rightarrow \bn{T}$ with $\phi_u(\bl{T'})\seq C^u_{\alpha_u}$.

Now, given $\sigma\in \bn{T'}$,  $\eta\in \bl{T'}$, and $u\in\{0,1\}$, define 
$$
e_{\sigma}^u=a_{\phi_u(\sigma)}^u\text{ and }g_{\eta}^u=b_{\phi_u(\eta)}^u.
$$
For each $u\in\{0,1\}$, since $\phi_u$ is a tree embedding it follows (using an argument nearly identical to the proof of Proposition \ref{prop:embedtree})  that  $(e_{\sigma}^u: \sigma\in \bn{T'})$ and $(g_{\eta}^u: \eta\in \bl{T'})$ form an agnostic $(T',\delta+\e)$-tree for $f$. 
By choice of $T'$, we can apply $P(t-1)$ to these trees. This yields, for each $u\in\{0,1\}$, a $(t-1,\delta)$-tree for $f$ with nodes $(c_{\sigma}^u: \sigma\in \bn{t-1})$, leaves $(d_{\eta}^u: \eta\in \bl{t-1})$, and values $(r^u_{\sigma}: \sigma\in \bn{t-1})$, such that for all $\eta\in \bl{t-1}$, $d^u_\eta$ is of the form $g^u_{\eta'}=b^u_{\phi_u(\eta')}$ for some $\eta'\in\bl{T'}$. By construction, and since $\phi_u(\bl{T'})\seq C^u_{\alpha_u}$, we have  conditions $(i)$ and $(ii)$. For condition $(iii)$, fix some arbitrary $\eta\in\bl{T'}$. Since  $\phi_u(\bl{T'})\seq C^u_{\alpha_u}$, we have
\[
\delta+\e\leq \left|f(x_{\semp},y_{0\sconc \phi_0(\eta)})-f(x_{\semp},y_{1\sconc\phi_1(\eta)})\right|
=\left|f(x_{\semp},b_{\phi_0(\eta)}^0)-f(x_{\semp}, b_{\phi_1(\eta)}^1)\right|\leq |\alpha_0-\alpha_1|+{\textstyle\frac{1}{K}}.
\]
So $|\alpha_0-\alpha_1|\geq \delta+\e-\frac{1}{K}$.\clqed\medskip

Claim 1 is all we will need to finish the proof. Since the statement is symmetric in $u\in\{0,1\}$, we may thus assume without loss of generality that $\alpha_0\leq\alpha_1$. So by Claim 1$(iii)$, 
\begin{equation}\label{split}
\alpha_1-\alpha_0\geq \delta+\e-{\textstyle\frac{1}{K}}.
\end{equation}

Define $a_{\semp}=x_{\semp}$ and $r_{\semp}=\alpha_0$. For each $\sigma\in \bn{t-1}$ and $u\in\{0,1\}$, set  $a_{u\sconc \sigma}=c_{\sigma}^u$ and $r_{u\sconc \sigma}=r^u_{\sigma}$.
For each $\eta\in \bl{t-1}$ and $u\in \{0,1\}$, set $b_{u\sconc \eta}=d^u_{\eta}$. We now check that this defines a $(t,\delta)$-tree for $f$ with nodes $(a_\sigma:\sigma\in \bn{t})$, leaves $(b_\eta:\eta\in \bl{t})$, and values $(r_\sigma:\sigma\in \bn{t})$.   Toward this end, assume we are given $\sigma\in \bn{t}$ and $\eta,\eta'\in \bl{t}$ satisfying $\sigma\conc 0\iseq \eta$ and $\sigma\conc 1\iseq \eta'$.  We want  to show $f(a_{\sigma},b_{\eta})\leq r_{\sigma}$ and $f(a_{\sigma},b_{\eta'})\geq r_\sigma+\delta$.

Assume first $\sigma\neq \emp$ and write $\sigma=u\conc \tau$ for some $u\in \{0,1\}$ and $\tau\in \bn{t-1}$. Then we can write $\eta=u\conc \nu$ and $\eta'=u\conc \nu'$ for some $\nu,\nu'\in \bl{t-1}$ such that $\tau\conc 0\iseq \nu$ and $\tau\conc 1\iseq \nu'$. Therefore
\begin{align*}
f(a_{\sigma},b_{\eta})&=f(c_{\tau}^u, d_{\nu}^u)\leq r_{\tau}^u=r_{\sigma} \text{ and }\\
f(a_{\sigma},b_{\eta'})&=f(c_{\tau}^u, d_{\nu'}^u)\geq r_{\tau}^u+\delta=r_{\sigma}+\delta.
\end{align*}

Assume now $\sigma=\emp$. Then $\eta=0\conc \tau$ and $\eta'=1\conc \tau'$ for some $\tau,\tau'\in \bl{t-1}$.  So by Claim 1$(ii)$, choice of $r_{\semp}$, and (\ref{split}), we have
\begin{align*}
f(a_{\semp},b_\eta) &= f(x_{\semp},d^0_\tau)\leq \alpha_0=r_{\semp} \text{ and }\\
f(a_{\semp},b_{\eta'})&=f(x_{\semp},d^1_{\tau'})\geq {\textstyle \alpha_1-\frac{1}{K}}\geq \alpha_0+\delta+\e-{\textstyle\frac{2}{K}}\geq r_{\semp}+\delta.
\end{align*}

Finally, by Claim 1$(i)$, each leaf $b_\eta$ of the above $(t,\delta)$-tree is of the form $y_{\eta'}$ for some $\eta'\in \bl{T}$. So we have proved $P(t)$ holds.
\end{proof}

\appendix

\section{Further results on trees and ladders}\label{sec:moreladders}

\subsection{Implications between ladders}\label{sec:ladderimps}

Recall from Subsection \ref{sec:ladders} that an $\e$-tight $(k,\delta)$-ladder is a special case of a uniform $(k,\delta)$-ladder, which itself is a special case of a $(k,\delta)$-ladder. Moreover, a uniform $(k,\delta)$-ladder is a special case of an agnostic $(k,\delta)$-ladder. In this section we state suitable converse implications between these notions. The first is an easy pigeonhole argument.

\begin{proposition}\label{prop:lad-unilad}
Fix $k\geq 1$, $\delta,\e>0$, and $f\colon X\times Y\to [0,1]$. Suppose $f$ admits a $(K,\delta+\epsilon)$-ladder, where $K=\lceil \epsilon\inv\rceil(k-1)+1$. Then $f$ admits a uniform $(k,\delta)$-ladder.
\end{proposition}
\begin{proof}
Suppose $x_1,\ldots,x_K\in X$ and $y_1,\ldots,y_K\in Y$ form a $(K,\delta+\e)$-ladder for $f$ with values $r_1,\ldots,r_K$. By pigeonhole, there is a subsequence $1\leq \ell_1<\ldots<\ell_k\leq K$ with $|r_{\ell_i}- r_{\ell_j}|\leq\e$ for all $i,j\in [k]$. It is now immediate that $(x_{\ell_1},\ldots,x_{\ell_k})$ and $(y_{\ell_1},\ldots,y_{\ell_k})$ form a uniform $(k,\delta)$-ladder for $f$ with value $r=\max\{r_{\ell_1},\ldots,r_{\ell_k}\}$.
\end{proof}

Next we prove a corresponding implication from agnostic ladders to uniform ladders. A statement to this effect is established by \cite[Proposition A.1]{CPC}. Our proof is nearly the same, except that we obtain a better bound through a more economical use of Ramsey's theorem. We also include the proof in order to draw a connection to tight ladders afterward. 

\begin{proposition}\label{prop:aglad-unilad}
Fix $k\geq 1$, $\delta,\epsilon>0$, and  $f\colon X\times Y\to [0,1]$. Suppose $f$ admits an agnostic $(K,\delta+\e)$-ladder, where $K=(2\lceil \e\inv\rceil)^{4\lceil \e\inv \rceil k}$. Then $f$ admits a uniform $(k,\delta)$-ladder.
\end{proposition}
\begin{proof}
Given integers $m,t\geq 1$, let $R_t(m)$ denote the minimal integer $n$ so that  any coloring of the edges of the complete graph on $n$ vertices with $t$ colors admits a complete monochromatic  graph on $m$ vertices. A well-known result of  Erd\"{o}s and Szekeres \cite{ErdSz} yields the bound $R_t(m)\leq t^{tm}$. In particular, note that $K\geq R_{2\lceil\e\inv\rceil}(2k)$. 

Now suppose $x_1,\ldots,x_K\in X$ and $y_1,\ldots,y_K\in Y$ form an agnostic $(K,\delta+\e)$-ladder for $f$. Let $\ell=\lceil\e\inv\rceil$ and let $\calI$ be a partition of $[0,1]$ consisting of intervals of length at most $\e$, with $|\calI|=\ell$. Consider the $2\ell$-coloring of the $2$-element subsets of $[K]$ where, given $1\leq i<j\leq K$, we color $\{i,j\}$  by the pair $(I,u)\in\calI\times\{0,1\}$ such that $f(x_i,y_j)\in I$ and $u=0$ if and only if $f(x_j,y_i)\leq f(x_i,y_j)-\delta-\e$ (so $u=1$ if and only if $f(x_j,y_i)\geq f(x_i,y_j)+\delta+\e$). By choice of $K$, we obtain  $a_1,\ldots, a_{2k}\in X$, $b_1,\ldots,b_{2k}\in Y$, and  $I\in\calI$ such that either:
\begin{enumerate}[$(i)$]
\item if $1\leq i<j\leq 2k$ then $f(a_i,b_j)\in I$ and $f(a_j,b_i)\leq f(a_i,b_j)-\delta-\e$, or
\item if $1\leq i<j\leq 2k$ then $f(a_i,b_j)\in I$ and $f(a_j,b_i)\geq f(a_i,b_j)+\delta+\e$.
\end{enumerate}
In case $(i)$, it is straightforward to check that $(a_1,a_3,\ldots,a_{2k-1})$ and $(b_2,b_4,\ldots,b_{2k})$ form a uniform $(k,\delta)$-ladder for $f$ with value $r=\inf I-\delta$. In case $(ii)$, it is straightforward to check that $(a_{2k},a_{2k-2},\ldots,a_2)$ and $(b_{2k-1},b_{2k-3},\ldots,b_{1})$ form a uniform $(k,\delta)$-ladder for $f$ with value $r=\sup I$. (These verifications are similar to the   proof of \cite[Proposition A.2]{CPC}).
\end{proof}

Note that the previous proof automatically obtains  tightness for ``half" of the resulting uniform ladder, witnessed by the interval $I$. One can obtain a fully tight ladder by modifying the proof to use a $2\lceil\e\inv\rceil^2$-coloring by triples $(I,J,u)$ with $f(x_i,y_j)\in I$, $f(x_j,y_i)\in J$, and $u$ as before. On the other hand, the $u$ value is  only used to reconcile the ``agnostic" aspect of the ladder, and thus is not needed if one starts with a non-agnostic ladder. We record these observations in the following proposition.

\begin{proposition}\label{prop:moreladimp}
Fix $k\geq 1$, $\delta,\e>0$, and $f\colon X\times Y\to[0,1]$. 
\begin{enumerate}[$(a)$]
\item If $f$ admits an agnostic $((2\lceil\e\inv\rceil)^{8\lceil\e\inv\rceil^2k},\delta+\e)$-ladder then it admits an $\e$-tight $(k,\delta)$-ladder.
\item If $f$ admits an $(\lceil\e\inv\rceil^{4\lceil\e\inv\rceil^2k},\delta+\e)$-ladder then it admits an $\e$-tight $(k,\delta)$-ladder.
\end{enumerate}
\end{proposition}

\subsection{Extracting trees from ladders}\label{sec:easyhodges}

In this section, we discuss the function-theoretic analogue of Theorem \ref{thm:hodges}$(1)$ (extracting a tree from a ladder). As mentioned in Subsection \ref{sec:prior}, results of this kind appear in  \cite[Lemma 8.1]{DaskGo}, \cite[Theorem 45 (first bullet)]{AndBen}, and \cite[Lemma A.11]{AADDF}. Each of these  involves some variation of a uniform ladder, formulated using various notions of ``threshold dimension" (see Definition \ref{def:threshdim}). Therefore,  the non-uniform version we state here is technically stronger. However, this has more to do with the fact that those sources do not consider the non-uniform case, rather than a crucial difference in the argument. Indeed, all of the proofs, including ours below, follow the same natural adaptation of Hodges' \cite{hodges} argument in the discrete case.  

\begin{theorem}\label{thm:optree}
For any $t\geq 1$ and $\delta>0$, if  $f\colon X\times Y\to [0,1]$ admits a $(2^{t},\delta)$-ladder,  then $f$ admits a $(t,\delta)$-tree.
\end{theorem}
\begin{proof}
Fix $\delta>0$ and $f\colon X\times Y\to [0,1]$.  Given an integer $t\geq 1$, let $P(t)$ be the following statement: \medskip

\noindent\textit{Statement of $P(t)$.} If $x_1,\ldots, x_{2^{t}}\in X$ and $y_1,\ldots, y_{2^{t}}\in Y$ form a $(2^{t},\delta)$-ladder for $f$ with values $r_1,\ldots,r_{2^{t}}\in [0,1]$, then there are maps $i\colon \bn{t}\to [2^{t}]$ and $j\colon \bl{t}\to [2^{t}]$ such that $(x_{i(\sigma)}:\sigma\in \bn{t})$ and $(y_{j(\eta)}:\eta\in \bl{t})$ form a $(t,\delta)$-tree for $f$ with values $(r_{i(\sigma)}:\sigma\in \bn{t})$. \medskip

We show $P(t)$ holds for all $t\geq 1$ by induction.\medskip

\textit{Base Case:}  Suppose $x_1,x_2\in X$ and $y_1,y_2\in Y$ form a $(2,\delta)$-ladder for $f$ with values $r_1,r_2\in[0,1]$. Set $i(\emp)=2$, $j(0)=1$, and $j(1)=2$. 
Then 
\begin{align*}
f(x_{i(\semp)},y_{j(0)}) &= f(x_2,y_1)\leq r_2=r_{i(\semp)} \text{ and }\\
f(x_{i(\semp)},y_{j(1)}) &=f(x_2,y_2)\geq r_2+\delta=r_{i(\semp)}+\delta.
\end{align*}
 So $x_{i(\semp)}$ and  $(y_{j(0)},y_{j(1)})$ form a $(1,\delta)$-tree for $f$ with value $r_{i(\semp)}$. Thus $P(1)$ holds.\medskip

\textit{Induction Step}: Fix $t>1$ and suppose  $P(t-1)$ holds. Set $T=2^{t}$ and $T'=2^{t-1}$. Suppose $x_1,\ldots, x_{T}\in X$ and $y_1,\ldots, y_{T}\in Y$ form a $(T,\delta)$-ladder for $f$ with values $r_1,\ldots,r_T\in [0,1]$. 

Since $T=2T'$, we have two $(T',\delta)$-ladders for $f$, namely:
\begin{enumerate}[$(1)$]
\item $(x_1,\ldots,x_{T'})$ and $(y_1,\ldots,y_{T'})$ with values $(r_1,\ldots,r_{T'})$, and
\item $(x_{T'+1},\ldots,x_{T})$ and $(y_{T'+1},\ldots,y_T)$ with values $(r_{T'+1},\ldots,r_T)$. 
\end{enumerate}
  Consequently, by $P(t-1)$,   there are maps 
  \[
  i_0\colon \bn{t-1}\to [T'],\quad j_0\colon \bl{t-1}\to [T'],\quad i_1\colon \bn{t-1}\to [T']+T'\quad\text{ and }\quad j_1\colon \bl{t-1}\to [T']+T'
  \]
  such that for each $u\in\{0,1\}$,  $(x_{i_u(\sigma)}:\sigma\in \bn{t-1})$ and $(y_{j_u(\eta)}:\eta\in \bl{t-1})$ form a $(t-1,\delta)$-tree for $f$ with values $(r_{i_u(\sigma)}:\sigma\in \bn{t-1})$. Thus for any $\sigma\in \bn{t-1}$ and $\eta,\eta'\in \bl{t-1}$ satisfying $\sigma \conc 0\iseq \eta$ and $\sigma\conc 1\iseq \eta'$, the following holds for each $u\in\{0,1\}$:
\begin{align}\label{line}
f(x_{i_u(\sigma)},y_{j_u(\eta)})\leq r_{i_u(\sigma)} \quad \text{and}\quad f(x_{i_u(\sigma)},y_{j_u(\eta')})\geq r_{i_u(\sigma)}+\delta.
\end{align}

Define $i\colon \bn{t}\to[T]$ so that $i(\emp)=T'+1$ and, for each $\sigma\in \bn{t-1}$ and $u\in\{0,1\}$, $i(u\conc\sigma)=i_u(\sigma)$. Define $j\colon \bl{t}\to [T]$ so that for each $\eta\in \bl{t-1}$ and $u\in\{0,1\}$, $j(u\conc \eta)=j_u(\eta)$. We show that $(x_{i(\sigma)}:\sigma\in \bn{t})$ and  $(y_{j(\eta)}:\eta\in \bl{t})$ form a $(t,\delta)$-tree for $f$ with values $(r_{i(\sigma)}:\sigma\in \bn{t})$. So suppose $\tau\in \bn{t}$ and $\mu,\mu'\in \bl{t}$ are such that $\tau\conc 0\iseq \mu$ and $\tau \conc 1\iseq \mu'$. We need to show 
\begin{align}\label{line2}
f(x_{i(\tau)},y_{j(\mu)})\leq r_{i(\tau)}\quad\text{and}\quad f(x_{i(\tau)},y_{j(\mu')})\geq r_{i(\tau)}+\delta.
\end{align}

Suppose first $\tau=\emp$. In this case, $\tau\conc 0=0\iseq \mu$ and $\tau\conc 1=1\iseq \mu'$, hence there exist $\eta,\eta'\in \bl{t-1}$ such that $\mu=0\conc \eta$ and $\mu'=1\conc \eta'$. Consequently, 
\[
j(\mu)=j_0(\eta)< T'+1=i(\tau)\leq j_1(\eta')=j(\mu').
\]
So (\ref{line2}) holds since $(x_1,\ldots, x_{T})$ and $(y_1,\ldots, y_{T})$ form a $(T,\delta)$-ladder for $f$ with values $(r_1,\ldots,r_T)$.

Suppose now $\tau\neq \emp$.  Then there exists $\sigma\in \bn{t-1}$,  $\eta,\eta'\in \bl{t-1}$, and $u\in \{0,1\}$ such that $\tau=u\conc \sigma$, $\mu=u\conc \eta$, $\mu'=u\conc \eta'$, $\sigma\conc 0\iseq \eta$, and $\sigma\conc 1\iseq \eta'$.  In this case, $i(\tau)=i_u(\sigma)$, $j(\mu)=j_u(\eta)$, and $j(\mu')=j_u(\eta')$, and thus (\ref{line2}) follows from (\ref{line}).
\end{proof}

By the previous proof, we also obtain the analogous implication between uniform ladders and trees, which is nearly identical to the second inequality in Eq. (11) of \cite[Lemma A.11]{AADDF}.

\begin{corollary}
For any $t\geq 1$ and $\delta>0$, if $f\colon X\times Y\to [0,1]$ admits a uniform $(2^t,\delta)$-ladder, then $f$ admits a uniform $(t,\delta)$-tree.
\end{corollary}

\section{Translation to learning theory}\label{sec:learning}

\subsection{Basic definitions}

In this subsection, we briefly explain how our combinatorial setting of  binary functions $f\colon X\times Y\to [0,1]$ is equivalent to the statistical learning theory setting of function classes $\cF\seq [0,1]^X$. We then define sequential fat-shattering dimension and fat-threshold dimension.

Given a function $f\colon X\times Y\to [0,1]$, we have the function class $\cF_f\coloneqq\{f_b:b\in Y\}$ where, given $b\in Y$, $f_b\colon X\to [0,1]$ denotes the fiber map sending $x$ to $f(x,b)$. Conversely, given a function class $\cF\seq [0,1]^X$, one can define the \emph{evaluation function} $E_{\cF}\colon X\times \cF\to [0,1]$ so that $E_{\cF}(x,f)=f(x)$. This is not a one-to-one correspondence since a function $f\colon X\times Y\to[0,1]$ may have repeated fibers. However, given $\cF\seq [0,1]^X$, the iterated class $\cF_{E_{\cF}}$ coincides with $\cF$.  Along the same lines, the dual of a function class $\cF\seq [0,1]^X$ can be canonically identified with $\cF_{(E_{\cF})^{opp}}$. 

We now state the definition of sequential fat-shattering dimension, which was first formulated  by Rakhlin, Sridharan, and Tewari \cite{RaSrTe0,RaSrTe}   in direct analogy to the Littlestone dimension of a discrete set system. The reader can check that, up to the notation introduced above, this definition is identical to that in \cite[Definition 7]{RaSrTe}. 

\begin{definition}\label{def:Ldim}
Fix $\delta>0$ and $\cF\seq [0,1]^X$. The \emph{sequential $\delta$-fat-shattering dimension of $\cF$} is the (possibly infinite) supremum over all $t$ such that $E_{\cF}$ admits a $(t,\delta)$-tree.
\end{definition}

We can also describe the \emph{dual sequential $\delta$-fat-shattering dimension} of $\cF\seq [0,1]^X$ as the supremum over all $t$ such that $(E_{\cF})^{opp}$ admits a $(t,\delta)$-tree.

Finally, we state the definition of fat-threshold dimension, as formulated in \cite[Definition A.10]{AADDF} and translated similarly. This notion will not be directly relevant to our main results, but we include it for the sake of completeness.

\begin{definition}\label{def:threshdim}
Fix $\delta>0$ and $\cF\seq [0,1]^X$. The \emph{$\delta$-fat-threshold dimension of $\cF$} is the (possibly infinite) supremum over all $k$ such that $E_{\cF}$ admits a uniform $(k,\delta)$-ladder.
\end{definition}


\subsection{Explanation of Theorem \ref{thm:DG}}\label{sec:DG}

In this subsection, we reconcile the result of Daskalakis and Golowich quoted in Subsection \ref{sec:prior} (namely, Theorem \ref{thm:DG}) with how it actually appears in \cite{DaskGo}. First, we state the definition of ``tight thresholds", quoting   \cite[Definition 8.1]{DaskGo} up to two innocuous changes, which we describe afterward.

\begin{definition}\label{def:tightDG}
Fix $k\geq 1$, $\alpha>2\beta>0$, and $\cF\seq [0,1]^X$. Then $\cF$ \emph{contains $k$ thresholds with margin $\alpha$ and tightness $\beta$} if there are $x_1,\ldots,x_k\in X$,   $f_1,\ldots,f_k\in\cF$, and $u,u'\in \R$ such that $|u-u'|\geq\alpha$ and, for all $i,j\in[k]$, if $i\leq j$ then $|f_j(x_i)-u|\leq \beta$, and if $i>j$ then $|f_j(x_i)-u'|\leq\beta$.
\end{definition}

The actual definition in \cite{DaskGo} restricts $u$ and $u'$ to $[0,1]$, whereas we allow $u,u'\in\R$ to avoid a boundary case conflict with our notion of tight ladders. Also, \cite{DaskGo} writes  $f_i(x_j)$ rather than $f_j(x_i)$, but this is equivalent up to reversing the order of indices, as in  Remark \ref{rem:uniformlad-sym}.

We  now summarize the connection between tight thresholds and tight ladders with the following remark, which is immediate from the definitions.

\begin{remark}\label{rem:tightconnection}$~$
\begin{enumerate}[$(a)$]
\item Given $k\geq 1$ and $\delta,\e>0$, a function $f\colon X\times Y\to [0,1]$ admits an $\e$-tight $(k,\delta)$-ladder if and only if $\cF_f$ contains $k$ thresholds with margin $\delta+\e$ and tightness $\frac{\e}{2}$ for which the witnesses $u,u'\in\R$ satisfy $u'<u$.
\item Fix $k\geq 1$, $\alpha>2\beta>0$, and $\cF\seq [0,1]^X$. Suppose $\cF$ contains $k$ thresholds with margin $\alpha$ and tightness $\beta$, witnessed by $u,u'\in\R$. Then either $u'<u$, in which case $E_{\cF}$ admits a $2\beta$-tight $(k,\alpha-2\beta)$-ladder, or $u<u'$, in which case $E_{\cF}$ admits a $2\beta$-tight $(k-1,\alpha-2\beta)$-ladder. 
\end{enumerate}
\end{remark}

Finally, we reconcile Theorem \ref{thm:DG} with \cite[Lemma 8.2]{DaskGo}.  In particular, using Definition \ref{def:Ldim} and Remark \ref{rem:tightconnection}, one can translate \cite[Lemma 8.2]{DaskGo} to a statement identical to Theorem \ref{thm:DG}, except with $4\delta+4\e$ instead of $2\delta+\e$. However, the increase from $2\delta$ to $4\delta$ is only due to the fact that the proof of \cite[Claim 8.3]{DaskGo} fixes $\alpha\geq 4\eta>0$, and  at a certain point replaces $\alpha/2-\eta$ by the lower bound $\alpha/4$. If one instead tracks $\alpha/2-\eta$ through their argument, then the resulting scale  is $2\delta+3\e$ which, for simplicity, we have written as $2\delta+\e$ in Theorem \ref{thm:DG}.\footnote{As in the first paragraph of the proof of Theorem \ref{thm:tight}, this is technically a different statement, but only in a way that affects the absolute constant in the bound.}

 \subsection{Explanation of Theorem \ref{thm:AB}}\label{sec:AB}
In this section, we reconcile the result of Anderson and Benedikt \cite{AndBen} quoted in Subsection \ref{sec:prior} (namely, Theorem \ref{thm:AB}) with the second bullet of \cite[Theorem 45]{AndBen}. This result states an implication from bounded ``$\delta$-threshold dimension" \cite[Definition 41]{AndBen} to bounded sequential fat-shattering dimension. Via Definition \ref{def:Ldim}, sequential fat-shattering dimension translates directly to $(t,\delta)$-trees; and \cite[Definition 41]{AndBen} translates to agnostic $(k,\delta)$-ladders in a similarly straightforward way. This yields Theorem \ref{thm:AB} exactly as stated except  that the bound is not made explicit and, strictly speaking, the formulation in \cite{AndBen} there has the parameter $t$ in Theorem \ref{thm:AB} also depending on $\delta$. However, an analysis of their proof shows that, with some minor tightening,  it yields a bound of the form 
\[
t\leq \frac{\lceil (2\e)\inv\rceil^k -1}{\lceil (2\e)\inv\rceil-1}\leq O((1/\epsilon)^k).
\]

\subsection{Further remarks}\label{rem:finalremarks}

The following are some remarks promised  earlier  in the paper.
In particular, we clarify the exact  form of our first main result (Theorem \ref{thm:hodgesfn1}) that one can obtain from the previous work of Anderson and Benedikt (Theorem \ref{thm:AB}) and of Daskalakis and Golowich (Theorem \ref{thm:DG}). We also justify the claim made in Problem \ref{prob:open1} related to deriving a dual sequential fat-shattering bound from Theorem \ref{thm:AB}. 

First, recall that Theorem \ref{thm:hodgesfn1} concerns extracting ladders from trees. On the other hand, Theorem \ref{thm:DG} extracts uniform ladders from trees and, as we will explain below, Theorem \ref{thm:AB} is also about this process. Thus these results are better compared to Corollary \ref{cor:hodgesfn1}, which shows that a $(\lceil\e\inv\rceil4^k,2\delta+\e)$-tree implies a uniform $(k,\delta)$-ladder. 

We now compare Corollary \ref{cor:hodgesfn1} to Theorem \ref{thm:AB}. The proof in \cite{AndBen} of the latter result  directly constructs an agnostic ladder from a tree. Thus, to obtain Corollary \ref{cor:hodgesfn1}, one must  apply an implication such as  Proposition \ref{prop:aglad-unilad} after the fact. Combining these bounds, we obtain Corollary \ref{cor:hodgesfn1} with the weaker bound $2^{(1/\e)^{O(k/\e)}}$. By running this through the proof of Corollary \ref{cor:dualbound}, we see that Theorem \ref{thm:AB} yields the same triple-exponential bound on dual sequential fat-shattering implied by Theorem \ref{thm:DG} (as claimed in Problem \ref{prob:open1}).

Finally, the above weaker bound for  Corollary \ref{cor:hodgesfn1} also matches what  one gets as a direct consequence of Theorem \ref{thm:DG} (which extracts tight ladders from trees).  Moreover, there is a structural similarity in the proofs of both Theorems \ref{thm:AB} and \ref{thm:DG}. In particular, as discussed after Theorem \ref{thm:tight}, the proof of Theorem \ref{thm:DG}  first uses a large tree to extract a configuration similar to an agnostic ladder (with additional tightness), and then applies the same multicolor Ramsey argument behind Proposition \ref{prop:aglad-unilad} to obtain a uniform tight ladder.

\section{Shelah 2-rank in continuous logic}\label{ss:2rankcont}

In this section, we discuss 2-rank in continuous theories, and note that it corresponds to our notion of uniform trees (Definition \ref{def:uniformtree}).  This section is for readers familiar with model theory, and will contain undefined model-theoretic terminology.  We refer the reader to \cite{BBHU} and \cite{shelah} for more background.

Stability in the context of continuous logic has deep roots in functional analysis. A brief history is discussed at the end of Section 4 of \cite{BBHU}, which also contains the following continuous logic analogue of Shelah 2-rank.

\begin{definition}\label{def:2rankdef}
Let $T$ be a complete first-order theory in continuous logic, and let $\calU$ be a monster model of $T$.   Let $p(x)$ be a partial type over a small subset of $\calU$, and suppose $\varphi_1(x,y)$ and $\varphi_2(x,y)$ are formulas such that the conditions $\varphi_1(x,y)=0$ and $\varphi_2(x,y)=0$ are contradictory. For ordinals $\alpha$, we inductively define $R(p,\varphi_1,\varphi_2,2)\geq\alpha$ as follows.
\begin{enumerate}[\hspace{5pt}$\ast$]
\item $R(p,\varphi_1,\varphi_2,2)\geq 0$ if and only if $p$ is consistent.
\item For a limit  $\lambda$, $R(p,\varphi_1,\varphi_2,2)\geq \lambda$  if and only if  $R(p,\varphi_1,\varphi_2,2)\geq \alpha$ for all $\alpha<\lambda$.
\item $R(p,\varphi_1,\varphi_2,2)\geq \alpha+1$ if  and only if there is some $b\in\calU^y$ and, for each $i\in\{1,2\}$, a  consistent type $p_i$ extending $p\cup\{\varphi_i(x,b)=0\}$ such that $R(p_i,\varphi_1,\varphi_2,2)\geq \alpha$. 
\end{enumerate}
\end{definition}

In \cite[Section 2]{BYscat} (which takes place in the broader setting of ``compact abstract theories"), global stability is shown to be equivalent to finiteness of  these ranks where $p$ is the partial type $\{d(x,x)=0\}$, and stable formulas are defined locally via this rank. While various equivalences are established there, the order property is not directly discussed. The order property version of stability for a continuous formula is defined by Ben Yaacov and Usvyatsov in \cite{BYU} and, although this source does not mention 2-rank, they establish equivalences that connect to 2-rank via \cite{BYscat}.

For local stability in  discrete logic, one considers the 2-rank $R(\{x=x\},\varphi,\neg \varphi,2)$.  The natural choice for a continuous formula is $R(\{d(x,x)=0\}, \varphi\dotminus r,s\dotminus\varphi,2)$ for $r<s$ (recall that $u\dotminus v=0$ is equivalent to $u\leq v$). Following in the same vein, we see that this choice of 2-rank is unbounded if and only if for all $t$, we can find $(b_{\sigma}:\sigma\in \bn{t})$ so that for each $\eta\in \bl{t}$, the type
$$
p_{\eta}=\{\varphi(x,b_{\sigma})\leq r: \sigma\conc 0\iseq \eta\}\cup \{\varphi(x,b_{\sigma})\geq s: \sigma\conc 1\iseq \eta\}
$$
is consistent. For a fixed $t$, realizing these types produces a sequence $(a_\eta:\eta\in\bl{t})$ which, together with $(b_\sigma:\sigma\in \bn{t})$, forms a uniform $(t,s-r)$-tree for $\varphi^{opp}(y,x)$ with value $r$. Note that the appearance of $\varphi^{opp}$ is because of our particular indexing convention for trees (see the footnote prior to Definition \ref{def:hodgestree}).

\section*{Acknowledgments}

 \subsection*{Humans} The authors thank Aaron Anderson for comments on a preliminary draft.

\subsection*{AI} ChatGPT was used for proofreading and for finding several  relevant and useful results in the literature. It also made the following mathematical contributions:

\begin{enumerate}[$(1)$]
\item Proposition \ref{prop:GPT1} was provided by ChatGPT upon direct request. 

\item Our original proof of Theorem \ref{thm:hodgesfn1} only established the uniform analogue in Corollary \ref{cor:hodgesfn1uni}, and only produced a bound of $6\cdot 4^k-2$. We asked ChatGPT if the bound could be improved, and it suggested changing the inductive scheme to its present two-parameter form, which yields the better bound $\binom{2k}{k}-1\sim 4^k/\sqrt{\pi k}$. Moreover, ChatGPT noted that the uniformity assumptions were only needed in our proof due to an imprecision in our original formulation of the conditions labeled (I) and (II). Thus we were able to remove these assumptions with only  minor revisions. It is worth noting that our application of Theorem \ref{thm:hodgesfn1} to dual sequential fat-shattering (Corollary \ref{cor:dualbound}) only requires the uniform case of Corollary \ref{cor:hodgesfn1uni}. However, the present form of Theorem \ref{thm:hodgesfn1} is crucial for obtaining polynomial bounds in several results of our companion paper \cite{CT-QSAR}.

\item Theorems \ref{thm:hodgesfn1} and \ref{thm:optree} were originally part of an early draft of our companion paper \cite{CT-QSAR}, which claimed that a quantitative account of the Shelah-Hodges correspondence for functions did not exist in the literature.  After circulating this draft, Anderson pointed us to his previous work with Benedikt \cite{AndBen}, which compelled us to query ChatGPT for a more extensive literature search. This led us  to the even earlier work of Daskalakis and Golowich \cite{DaskGo} and, in particular, the open question of recovering the bound on extracting tight ladders from trees (claimed in \cite{JuKiTe}). As explained after Theorem \ref{thm:tight}, Daskalakis and Golowich fill the gap in \cite{JuKiTe} by first proving a ``half tight" extraction \cite[Claim 8.4]{DaskGo}, followed by a multicolored Ramsey argument, which adds an exponential to the bound  in \cite{JuKiTe}. We suspected that our methods could be modified to allow for a fully tight extraction and avoid the use of multicolored Ramsey numbers altogether. So we gave ChatGPT the latest draft of our paper and asked it to produce a proof. This draft included a short description of the obstacles such a proof would need to overcome. 
ChatGPT successfully generated an argument following the  induction scheme and rough two-case structure of our original proof of Theorem \ref{thm:hodgesfn1}. The tools used in the proof (e.g. Lemmas \ref{lem:ramsey} and \ref{lem:ramseyleaves}) were already present in our original draft prior to consulting ChatGPT. To streamline and elucidate the proof, we isolated the intermediate proto-ladder configuration, and separated the two technical lemmas comprising the argument. 
\end{enumerate}

All proofs and examples generated by ChatGPT were carefully checked 
 and thoroughly rewritten by the authors.

\end{document}